\documentclass[11pt,a4paper]{article}

\usepackage{titlesec}
\usepackage{a4wide}
\usepackage{graphicx}
\usepackage{float}
\usepackage{amssymb}
\usepackage{amsmath}
\usepackage{amsthm}
\usepackage{appendix}
\usepackage{color}
\usepackage{array}
\usepackage{eucal}
\usepackage{mathrsfs}
\usepackage{esint}
\usepackage{lmodern}
\usepackage{tikz}
\usepackage{bbm}
\usepackage{hyperref}
\usepackage{geometry}
\usepackage{changepage}
\changepage{0pt}{}{}{}{}{0pt}{}{0pt}{10pt}
\usepackage[numbers]{natbib}
\hypersetup{
pdfpagemode=UseThumbs,
pdftoolbar=true,        % show Acrobat's toolbar?
pdfmenubar=true,        % show Acrobats menu?
pdffitwindow=false,     % window fit to page when opened
pdfstartview={Fit},    % fits the width of the page to the window
pdftitle={Maxwell's equations with hypersingularities at a conical plasmonic tip},    % title
pdfauthor={A.-S. Bonnet-Ben Dhia, L. Chesnel, M. Rihani},     % author
pdfsubject={},  % subject of the document
pdfcreator={A.-S. Bonnet-Ben Dhia, L. Chesnel, M. Rihani},   % creator of the document
pdfproducer={A.-S. Bonnet-Ben Dhia, L. Chesnel, M. Rihani}, % producer of the document
pdfkeywords={}, % list of keywords
pdfnewwindow=true,      % links in new window
colorlinks=true,       % false: boxed links; true: colored links
linkcolor=magenta,          % color of internal links
citecolor=red,        % color of links to bibliography
filecolor=cyan,      % color of file links
urlcolor=blue           % color of external links
}

\usetikzlibrary{shapes.misc}
\tikzset{cross/.style={cross out, draw=blue, minimum size=1*(#1-\pgflinewidth), inner sep=0pt, outer sep=0pt},
cross/.default={6pt}}

\def\u{\boldsymbol{u}}
\def\v{\boldsymbol{v}}
\def\w{\boldsymbol{w}}
\def\E{\boldsymbol{E}}

\def\H{\boldsymbol{H}}
\def\Hspace{\boldsymbol{\mrm{H}}}
\def\Lspace{\boldsymbol{\mrm{L}}}

\def\J{\boldsymbol{J}}

\def\C{\mathbb C}		% complexes
\newcommand{\dsp}{\displaystyle}

\newcommand{\eps}{\varepsilon}
\newcommand{\om}{\omega}
\newcommand{\Om}{\Omega}
\newcommand{\mrm}[1]{\mathrm{#1}}

\newcommand{\Cplx}{\mathbb{C}}
\newcommand{\N}{\mathbb{N}}
\newcommand{\R}{\mathbb{R}}

\renewcommand{\div}{\mrm{div}}
\newcommand{\mL}{\mrm{L}}
\newcommand{\mH}{\mrm{H}}
\newcommand{\mV}{\mrm{V}}
\newcommand{\mX}{\boldsymbol{\mrm{X}}}

\newcommand{\mZ}{\boldsymbol{\mrm{Z}}}

\newcommand{\coker}{\mrm{coker}\,}

\newcommand{\curl}{\boldsymbol{\mrm{curl}}\,}

\newcommand{\fintH}{\fint}
\newtheorem{theorem}{Theorem}[section]
\newtheorem{lemma}[theorem]{Lemma}
\newtheorem{remark}[theorem]{Remark}

\newtheorem{proposition}[theorem]{Proposition}

\everymath{\displaystyle}

\begin{document}

~\vspace{0.0cm}
\begin{center}
{\sc \bf\LARGE  Maxwell's equations with hypersingularities\\[6pt] at a conical plasmonic tip}
\end{center}
\begin{center}
\textsc{Anne-Sophie Bonnet-Ben Dhia}$^1$, \textsc{Lucas Chesnel}$^2$, \textsc{Mahran Rihani}$^{1,2}$\\[16pt]
\begin{minipage}{0.96\textwidth}
{\small
$^1$ Laboratoire  Poems, CNRS/INRIA/ENSTA Paris, Institut Polytechnique de Paris, 828 Boulevard des Mar\'echaux, 91762 Palaiseau, France;\\
$^2$ INRIA/Centre de math\'ematiques appliqu\'ees, \'Ecole Polytechnique,  Institut Polytechnique de Paris, Route de Saclay, 91128 Palaiseau, France.\\[10pt]
E-mails: \texttt{anne-sophie.bonnet-bendhia@ensta-paris.fr}, \texttt{lucas.chesnel@inria.fr}, \texttt{mahran.rihani@ensta-paris.fr}\\[-14pt]
\begin{center}
(\today)
\end{center}
}
\end{minipage}
\end{center}
\vspace{0.4cm}

\noindent\textbf{Abstract.} In this work, we are interested in the analysis of time-harmonic  Maxwell's equations in presence of a conical tip of a material with negative dielectric constants. When these constants belong to some critical range, the electromagnetic field exhibits strongly oscillating singularities at the tip which have infinite energy. Consequently Maxwell's equations are not well-posed in the classical $\mrm{L}^2$ framework. The goal of the present work is to provide an appropriate functional setting for 3D Maxwell's equations when the dielectric permittivity (but not the magnetic permeability) takes critical values. Following what has been done for the 2D scalar case, the idea is to work  in weighted Sobolev spaces,  adding to the space  the so-called outgoing propagating singularities. The analysis requires new results of scalar and vector potential representations of singular fields. The outgoing behaviour is selected via the limiting absorption principle. \\


\section{Introduction}

For the past two decades, the scientific community has been particularly interested in the study of Maxwell's equations in the unusual case where the dielectric permittivity $\eps$ is a real-valued sign-changing function. There are several motivations to this which are all related to  spectacular progress in physics. Such sign-changing $\eps$ appear for example in the field of plasmonics \cite{BaDE03,Maier07,BVNMRB08}. The existence of surface plasmonic waves is mainly due to the fact that, at optical frequencies, some metals like silver or gold have an $\eps$  with a small imaginary part and a negative real part. Neglecting the imaginary part, at a given frequency, one is led to consider a real-valued $\eps$ which is negative in the metal and positive in the air around the metal. A second more prospective motivation concerns the so-called metamaterials, whose micro-structure is designed so that their effective electromagnetic constants may have a  negative real part and a small imaginary part in some frequency ranges \cite{SmiPenWil04,Sih07,SalEng06}. Let us emphasize that for such metamaterials not only the dielectric permittivity $\eps$ may become negative but the magnetic permeability $\mu$ as well. At the interface between dielectrics and negative-index metamaterials, one can observe a negative refraction phenomenon which opens a lot of exciting prospects. Finally let us mention that negative $\eps$ also appear in plasmas, together with strong anisotropic effects.  But we want to underline a main difference between  plasmas and the previous applications. In the case of plasmonics and metamaterials, $\eps$ is sign-changing but does not vanish (and similarly for $\mu$), while in plasmas, $\eps$ vanishes on some particular surfaces, leading to the phenomenon of hybrid resonance (see \cite{DeIW14,NiCD19}). The theory developed in the present paper does no apply to the case where $\eps$ vanishes.\\
\newline
The goal of the present work is to study the Maxwell's system in the case where $\eps$, $\mu$ change sign but do not vanish. In case of invariance with respect to one variable, the analysis of time-harmonic Maxwell's problem leads to consider the 2D scalar Helmholtz equation
\[
\div\left(\frac{1}{\eps}\nabla \varphi\right)+\om^2\mu\varphi=f.
\]
Here $f$ denotes the source term and the unknown $\varphi$ is a component of the magnetic field. For this scalar equation, only the change of sign of $\eps$ matters because roughly speaking, the term involving $\mu$ is compact (or locally compact in freespace). In the particular case where $\eps$ takes constant values $\eps_+>0$ and $\eps_-<0$ in two subdomains separated by a a curve $\Sigma$, the results are quite complete \cite{BoChCi12}. If $\Sigma$ is smooth (of class ${\mathcal C}^1$), the equation has the same properties in the $\mH^1$ framework as in the case of positive coefficients, except when the contrast  $\kappa_{\eps}:=\eps_-/\eps_+$ takes the particular value $-1$. One way to show this consists in finding an appropriate operator $\mathbb T$ such that the coercivity of the variational formulation is restored when testing with functions of the form $\mathbb T\varphi'$ (instead of $\varphi'$). This approach is called the $\mathbb T$-coercivity technique. 
When $\kappa_{\eps}=-1$, Fredholmness is lost in $\mH^1$ but some results can be established in some weighted Sobolev spaces where the weight is adapted to the shape of $\Sigma$ \cite{Ola95,Nguy16,Pank19}. The picture is quite different when $\Sigma$ has corners. For instance, in the case of a polygonal curve $\Sigma$, Fredholmness in $\mH^1$ is lost not only for $\kappa_{\eps}=-1$ but for a whole interval of values of $\kappa_{\eps}$ around $-1$. We name this interval the critical interval. The smaller the angle of the corners, the larger the critical interval is. In fact, we can still find a solution in that case but this solution has a strongly singular behaviour at the corners in $r^{i\eta}$ where $r$ is the distance to the corner and $\eta$ is a real coefficient. In particular, this hypersingular solution does not belong to $\mH^1$. It has been shown 
that Fredholmness can be recovered in an appropriate unusual framework \cite{BonCheCla13} which is obtained by adding a singular function to a Kondratiev weighted Sobolev space of regular functions. The proof requires to adapt Mellin techniques in Kondratiev  spaces \cite{Kond67} to an equation which is not elliptic due to the change of sign of $\eps$ (see \cite{DaTe97} for the first analysis). From a physical point of view, the singular\footnote{From now on, we simply write ``singular'' instead of ``hypersingular''.} function corresponds to a wave which propagates towards the corner, without never reaching it because its group velocity tends to zero with the distance to the corner \cite{BCCC16,HeKa18,HeKa20}. In the literature, this wave which is trapped by the corner is commonly referred to as a black-hole wave. It leads to a strange phenomenon of leakage of energy while only non-dissipative materials are considered.\\
\newline
The objective of this article is to extend this type of results to 3D Maxwell's equations. The case where the contrasts in $\eps$ and $\mu$ do not take critical values has been considered in \cite{BoCC14}. Using the $\mathbb T$-coercivity technique, a Fredholm property has been proved for Maxwell equations in a classical functional framework as soon as two scalar problems (one for $\eps$ and one for $\mu$) are well-posed in $\mH^1$. The case where these problems satisfy a Fredholm property in $\mH^1$ but with a non trivial kernel has also been treated in \cite{BoCC14}. Let us finally mention \cite{NgSi19} where different types of results have been established for a smooth inclusion of class ${\mathscr C}^1$. In the present work, we consider a 3D configuration with an inclusion of material with a negative dielectric permeability $\eps$. We suppose that this inclusion has a tip at which singularities of the electromagnetic field exist. The objective is to combine Mellin analysis in Kondratiev  spaces with the $\mathbb T$-coercivity technique to derive an appropriate functional framework for Maxwell's equations when the contrast $\kappa_{\eps}$ takes critical values (but not the contrast in $\mu$). We emphasize that due to the non standard singularities we have to deal with, the results we obtain are quite different from the ones existing for classical Maxwell's equations with positive materials in non smooth domains \cite{BiSo87b,Cost90,BiSo94,CoDN99,CoDa00}.\\
\newline
The outline is as follows. In the remaining part of the introduction, we present some general notation. In Section \ref{sec-assumptions}, we describe the assumptions made on the dielectric constants $\eps$, $\mu$. Then we propose a new functional framework for the problem for the electric field and show its well-posedness in Section \ref{SectionChampE}. Section \ref{SectionChampH} is dedicated to the analysis of the problem for the magnetic field. We emphasize that due to the assumptions made on $\eps$, $\mu$ (the contrast in $\eps$ is critical but the one in $\mu$ is not), the studies in sections  \ref{SectionChampE} and \ref{SectionChampH} are quite different. We give a few words of conclusion in Section \ref{SectionConclusion} before presenting technical results needed in the analysis in two sections of appendix. The main outcomes of this work are Theorem \ref{MainThmE} (well-posedness for the electric problem) and Theorem \ref{MainThmH} (well-posedness for the magnetic problem).\\
\newline
All the study will take place in some domain $\Om$ of $\R^3$. More precisely, $\Om$ is an open, connected and bounded subset of $\R^3$ with a Lipschitz-continuous boundary $\partial\Om$. Once for all, we make the following assumption:\\
\newline
\textbf{Assumption 1.} \textit{The domain $\Om$ is simply connected and $\partial\Om$ is connected.}\\
\newline
When this assumption is not satisfied, the analysis below must be adapted (see the discussion in the conclusion). For some $\om\ne0$ ($\om\in\R$), the time-harmonic Maxwell's equations are given by
\begin{equation}\label{Eqs Maxwell}
\curl\E-i\om\,\mu\,\H = 0\qquad\mbox{ and }\qquad\curl\H+i\om\,\eps\,\E = \J\,\mbox{ in }\Om.
\end{equation}
Above $\E$ and $\H$ are respectively the electric and magnetic components of the electromagnetic field. The source term $\J$ is the current density. We suppose that the medium $\Om$ is surrounded by a perfect conductor and we impose the boundary conditions
\begin{equation}\label{CL Maxwell}
\E\times\nu=0\qquad\mbox{ and }\qquad\mu\H\cdot\nu=0\,\mbox{ on }\partial\Om,
\end{equation}
where $\nu$ denotes the unit outward normal vector field to $\partial\Om$. The dielectric permittivity $\eps$ and the magnetic permeability $\mu$ are real valued functions which belong to $\mrm{L}^{\infty}(\Om)$, with $\eps^{-1},\,\mu^{-1}\in \mrm{L}^{\infty}(\Om)$ (without assumption of sign). Let us introduce some usual spaces in the study of Maxwell's equations:
$$
\begin{array}{rcl}
\Lspace^2(\Omega)&:=&(\mL^2(\Omega))^3\\
\mH^1_{0}(\Omega)&:=&\{\varphi\in\mH^1(\Om)\,|\,\varphi=0\mbox{ on }\partial\Om\}\\
\mH^1_{\#}(\Omega)&:=&\{\varphi\in\mH^1(\Om)\,|\,\int_{\Om}\varphi\,dx=0\}\\
\dsp \Hspace(\curl) &:=& \dsp \{ \boldsymbol{H}\in \Lspace^2(\Omega) \,|\, \curl \boldsymbol{H}\in\Lspace^2(\Omega)\}\\[2pt]
\dsp \Hspace_N(\curl) &:=& \dsp \{ \boldsymbol{E}\in \Hspace(\curl) \,|\, \boldsymbol{E}\times\nu=0 \mbox { on } \partial\Omega\}
\end{array}
$$
and for $\xi\in \mrm{L}^{\infty}(\Om)$:
$$
\begin{array}{rcl}
\mX_T(\xi) & :=& \left\{\boldsymbol{H}\in \Hspace(\curl)\,|\,\div(\xi\boldsymbol{H})=0,\,\xi\boldsymbol{H}\cdot\nu=0 \mbox{ on }\partial\Om \right\}\\[2pt]
\mX_N(\xi) & := &\left\{\boldsymbol{E}\in \Hspace_N(\curl)\,|\,\div(\xi\boldsymbol{E})=0\right\}.
\end{array}
$$
We denote indistinctly by $(\cdot,\cdot)_{\Om}$ the classical inner products of $\mL^2(\Om)$ and $\Lspace^2(\Om)$. Moreover, $\|\cdot\|_{\Om}$ stands for the corresponding norms. We endow the spaces $\Hspace(\curl)$, $\Hspace_N(\curl)$, $\mX_T(\xi)$, $\mX_N(\xi)$ with the norm 
\[
\|\cdot\|_{\Hspace(\curl)}:=(\|\cdot\|^2_{\Om}+\|\curl\cdot\|^2_{\Om})^{1/2}.
\]
Let us recall a well-known property for the particular spaces $\mX_T(1)$ and $\mX_N(1)$ (cf. \cite{Webe80,AmrBerDau98}).
\begin{proposition}\label{PropoEmbeddingCla}
Under Assumption 1, the embeddings of $\mX_T(1)$ in $\Lspace^2(\Om)$ and of $\mX_N(1)$ in $\Lspace^2(\Om)$ are compact. And there is a constant $C>0$ such that 
\[
\|\u\|_{\Om}\le C\,\|\curl\u\|_{\Om},\qquad \forall\u\in\mX_T(1)\cup\mX_N(1).
\]
Therefore, in $\mX_T(1)$ and in $\mX_N(1)$, $\|\curl\cdot\|_{\Om}$ is a norm which is equivalent to $\|\cdot\|_{\Hspace(\curl)}$.
\end{proposition}

\section{Assumptions for the dielectric constants $\eps$, $\mu$}\label{sec-assumptions}

In this document, for a Banach space $\mrm{X}$, $\mrm{X}^{\ast}$ stands for the topological antidual space of $\mrm{X}$ (the set of continuous anti-linear forms on $\mrm{X}$).
\\In the analysis of the Maxwell's system (\ref{Eqs Maxwell})-(\ref{CL Maxwell}), the properties of two scalar operators associated respectively with $\eps$ and $\mu$ play a key role. Define $A_{\eps}:\mH^1_0(\Om)\to(\mH^1_0(\Om))^{\ast}$ such that 
\begin{equation}\label{DefAeps}
\langle A_{\eps}\varphi,\varphi'\rangle=\int_{\Om}\eps\nabla\varphi\cdot\nabla\overline {\varphi'}\,dx,\qquad \forall \varphi,\varphi'\in\mH^1_0(\Om)
\end{equation}
and $A_{\mu}:\mH^1_{\#}(\Om)\to(\mH^1_{\#}(\Om))^{\ast}$ such that 
\[
\langle A_{\mu}\varphi,\varphi'\rangle=\int_{\Om}\mu\nabla\varphi\cdot\nabla\overline{\varphi'}\,dx,\qquad \forall \varphi,\varphi'\in\mH^1_{\#}(\Om).
\]
\textbf{Assumption 2.} \textit{We assume that $\mu$ is such that $A_{\mu}:\mH^1_{\#}(\Om)\to(\mH^1_{\#}(\Om))^{\ast}$ is an isomorphism.}\\
\newline
Assumption 2 is satisfied in particular if $\mu$ has a constant sign (by Lax-Milgram theorem). We underline however that we allow $\mu$ to change sign (see in particular \cite{CosSte85,BonCiaZwo08,BoChCi12,BoCC14} for examples of sign-changing $\mu$ such that Assumption 2 is verified). The assumption on $\eps$, that will be responsible for the presence of (hyper)singularities, requires to consider a more specific configuration as explained below. 
\subsection{Conical tip and scalar (hyper)singularities}
\label{subsec-hypersingularities}
We assume that $\Om$ contains an inclusion of a particular material (metal at optical frequency, metamaterial, ...) located in some domain $\mathcal{M}$ such that $\overline{\mathcal{M}}\subset\Om$ ($\mathcal{M}$ like metal or metamaterial). We assume that $\partial\mathcal{M}$ is of class $\mathscr{C}^2$ except at the origin $O$ where $\mathcal{M}$ coincides locally with a conical tip. More precisely, there are $\rho>0$ and some smooth domain $\varpi$ of the unit sphere $\mathbb{S}^2:=\{x\in\R^3\,|\,|x|=1\}$ such that  $B(O,\rho)\subset \Omega$ and 
\[
\mathcal{M}\cap B(O,\rho)=\mathcal{K}\cap B(O,\rho)\qquad\mbox{ with }\ \mathcal{K}:=\{r\,\boldsymbol{\theta}\,|\,r>0,\,\boldsymbol{\theta}\in\varpi \}.
\]
Here $B(O,\rho)$ stands for the open ball centered at $O$ and of radius $\rho$. We assume that $\eps$ takes the constant value $\eps_-<0$ (resp. $\eps_+>0$) in $\mathcal{M}\cap B(O,\rho)$ (resp. $(\Om\setminus\overline{\mathcal{M}})\cap B(O,\rho)$). And we assume that the contrast $\kappa_{\eps}:=\eps_-/\eps_+<0$ and $\varpi$ (which characterizes the geometry of the conical tip) are such that there exist singularities of the form
\begin{equation}\label{defSing}
\mathfrak{s}(x)=r^{-1/2+i\eta}\Phi(\theta,\phi)
\end{equation}
satisfying  $\div(\eps\nabla\mathfrak{s})=0$ in $\mathcal{K}$
with $\eta\in\R,\eta\neq 0$. Here $(r,\theta,\phi)$ are the spherical coordinates associated with $O$ while $\Phi$ is a function which is smooth in $\varpi$ and in $\mathbb{S}^2\setminus\overline{\varpi}$. We emphasize that since the interface between the metamaterial and the exterior material is not smooth, singularities always exist at the conical tip. However, here we make a particular assumption on the singular exponent  which has to be of the form $-1/2+i\eta$ with  $\eta\in\R,\eta\neq 0$. Such singularities play a particular role for the operator $A_{\eps}$ introduced in (\ref{DefAeps}) because they are ``just'' outside $\mH^1$. More precisely, we have $\mathfrak{s}\notin\mH^1(\Om)$ but $r^\gamma\mathfrak{s}\in\mH^1(\Om)$ for all $\gamma>0$. With them, we can construct a sequence of functions $u_n\in\mH^1_0(\Om)$ such that 
\[
\forall n\in\N,\quad\|u_n\|_{\mH^1(\Om)}=1\qquad\mbox{ and }\qquad\lim_{n\to+\infty}\|\div(\eps\nabla u_n)\|_{(\mH^1_0(\Om))^{\ast}}+\|u_n\|_{\Om}=0.
\]
Then this allows one to prove that the range of $A_{\eps}:\mH^1_0(\Om)\to(\mH^1_0(\Om))^{\ast}$ is not closed (see \cite{BonDauRam99,BoChCi12,BonCheCla13} in 2D). Of course, for any given geometry, such singularities do not exist when $\kappa_{\eps}>0$ because we know that in this case $A_{\eps}:\mH^1_0(\Om)\to(\mH^1_0(\Om))^{\ast}$ is an isomorphism. On the other hand, when 
\begin{equation}\label{ConicalTip}
\varpi=\{(\cos\theta\cos\phi,\sin\theta\cos\phi,\sin\phi)\,|\,-\pi\le\theta\le\pi,\,-\pi/2\le\phi <-\pi/2+\alpha\}\mbox{ for some }\alpha\in(0;\pi)
\end{equation}
(the circular conical tip, see Figure \ref{fig-geom}), it can be shown that such $\mathfrak{s}$ exists for  $\kappa_{\eps}>-1$ (resp. $\kappa_{\eps}<-1$) and $|\kappa_{\eps}+1|$ small enough (see \cite{KCHWS14}) when $\alpha<\pi/2$ (resp. $\alpha>\pi/2$). 
\begin{figure}[!ht]
\centering
\includegraphics[width=6cm]{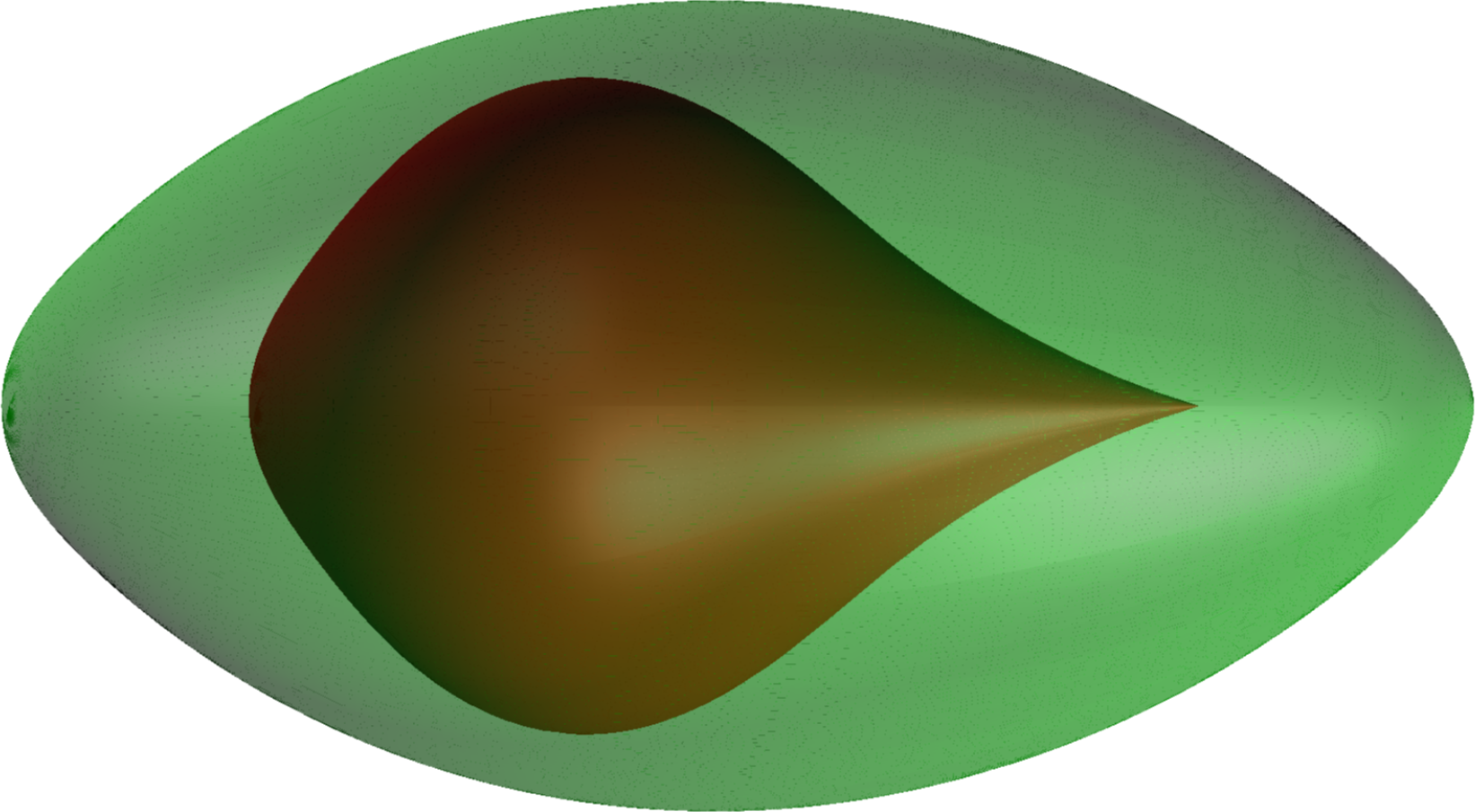}
\caption{The domain $\Om$ with the inclusion  $\mathcal{M}$ exhibiting a conical tip. \label{fig-geom}}
\end{figure}
For a general smooth domain $\varpi\subset\mathbb{S}^2$ and a given contrast $\kappa_{\eps}$, in order to know if such $\mathfrak{s}$ exists, one has to solve the spectral problem 
\begin{equation}\label{spectralPb}
\begin{array}{|l}
\mbox{Find }(\Phi,\lambda)\in\mH^1(\mathbb{S}^2)\setminus\{0\}\times\C\mbox{ such that }\\
\dsp\int_{\mathbb{S}^2}\eps\nabla_S\Phi\cdot\nabla_S\overline{\Phi'}\,ds 
= \lambda(\lambda+1)\int_{\mathbb{S}^2}\eps\Phi\,\overline{\Phi'}\,ds,\qquad \forall\Phi'\in\mH^1(\mathbb{S}^2),
\end{array}
\end{equation}
and see if among the eigenvalues some of them are of the form $\lambda=-1/2+i\eta$ with $\eta\in\R, \eta\neq 0$.  Above, $\nabla_S$ stands for the surface gradient. With a slight abuse, when $\eps$ is involved into integrals over $\mathbb{S}^2$, we write $\eps$ instead of $\eps(\rho\,\cdot)$. Note that since $\eps$ is real-valued,  if $\lambda=-1/2+i\eta$ is an eigenvalue, we have $\lambda(\lambda+1)=-\eta^2-1/4$, so that $\lambda=-1/2-i\eta$ is also an eigenvalue for the same eigenfunction. And since $\lambda(\lambda+1)\in\R$, we can find a corresponding eigenfunction which is real-valued. From now on, we assume that $\Phi$ in (\ref{defSing}) is real-valued. Let us mention that this problem of existence of singularities of the form (\ref{defSing}) is directly related to the problem of existence of essential spectrum for the so-called Neumann-Poincar\'e operator \cite{KaPS07,PePu17,BoZh19,HePe18}. A noteworthy difference with the 2D case of a corner in the interface is that several singularities of the form (\ref{defSing}) with different values of $|\eta|$  can exist in 3D \cite{KCHWS14}  (this depends on $\eps$ and on $\varpi$). To simplify the presentation, we assume that for the case of interest, singularities of the form (\ref{defSing}) exist for only one value of $|\eta|$. Moreover we assume that the quantity $\textstyle\int_{\mathbb{S}^2}\eps|\Phi|^2ds$ does not vanish. In this case, exchanging $\eta$ by $-\eta$ if necessary, we can set $\eta$ so that 
\begin{equation}\label{eq-intPhineq0}
\eta \int_{\mathbb{S}^2}\eps|\Phi|^2ds>0.
\end{equation}
For the 2D problem, it can be proved that the quantity corresponding to $\textstyle\int_{\mathbb{S}^2}\eps|\Phi|^2ds$ vanishes if and only if the contrast $\kappa_\eps$ coincides with a bound of the critical interval. We conjecture that this also holds in 3D. Note that when $\textstyle\int_{\mathbb{S}^2}\eps|\Phi|^2ds=0$, the singularities have a different form from (\ref{defSing}). To fix notations, we  set 
\begin{equation}\label{DefSinguEss}
s^{\pm}(x)=\chi(r)r^{-1/2\pm i\eta}\Phi(\theta,\phi)
\end{equation}
In this definition the smooth cut-off function $\chi$ is equal to one in a neighbourhood of $0$ and is supported in $[-\rho;\rho]$. In particular, we emphasize that $s^{\pm}$ vanish in a neighbourhood of $\partial\Om$.\\ 
\newline 
In order to recover Fredholmness for the scalar problem involving $\eps$, an important idea is too add one (and only one) of the singularities (\ref{DefSinguEss}) to the functional framework. From a mathematical point of view, working with the complex conjugation, it is obvious to see that adding $s^+$ or $s^-$ does not change the results. However physically one framework is more relevant than the other. More precisely, we will explain in \S\ref{ParagLimiting} with the limiting absorption principle why selecting $s^+$, with $\eta$ such that (\ref{eq-intPhineq0}) holds, together with a certain convention for the time-harmonic dependence, is more natural.

\subsection{Kondratiev functional framework}
In this paragraph, adapting what is done in \cite{BonCheCla13} for the 2D case, we describe in more details how to get a Fredholm operator for the scalar operator associated with $\eps$. For $\beta \in \mathbb{R}$ and $m \in \mathbb{N}$, let us introduce the weighted Sobolev (Kondratiev) space $\mV_\beta^m(\Om)$ (see \cite{Kond67}) defined as the closure of $\mathscr{C}^{\infty}_0(\overline{\Om}\setminus \{O\})$
for the norm
\[
\|\varphi\|_{\mV_\beta^m(\Om)}=\left(\sum_{|\alpha| \leq m} \|r^{|\alpha|-m+\beta} \partial_x^{\alpha}\varphi\|_{\mL^2(\Om)}^2\right)^{1/2}.
\]
Here $\mathscr{C}^{\infty}_0(\overline{\Om}\setminus \{O\})$ denotes the space of infinitely differentiable functions which are supported in $\overline{\Om}\setminus \{O\}$. We also denote $\mathring{\mV}_\beta^1(\Om)$ the closure of  $\mathscr{C}^{\infty}_0(\Om\setminus \{O\})$ for the norm
$\|\cdot\|_{\mV_\beta^1(\Om)}$. We have the characterisation
\[
\mathring{\mV}_\beta^1(\Om)=\{\varphi\in\mV_\beta^1(\Om)\,|\,\varphi=0\mbox{ on }\partial\Om\}.
\]
Note that using Hardy's inequality 
\[
\int_{0}^1\frac{|u(r)|^2}{r^2}\,r^2dr \le 4\,\int_{0}^1|u'(r)|^2\,r^2dr,\qquad\forall u\in\mathscr{C}^1_0[0;1),
\]
one can show the estimate $\|r^{-1}\varphi\|_{\Om}\le C\,\|\nabla\varphi\|_{\Om}$ for all 
$\varphi\in\mathscr{C}^{\infty}_0(\Om\setminus \{O\})$. This proves that $\mathring{\mV}_{0}^1(\Om)=\mH^1_0(\Om)$. Now set $\beta>0$. Observe that we have 
\[
\mathring{\mV}_{-\beta}^1(\Om)\subset \mH^1_0(\Om)\subset \mathring{\mV}_\beta^1(\Om)\qquad\mbox{ so that }\qquad(\mathring{\mV}^1_{\beta}(\Om))^{\ast}\subset (\mH^1_0(\Om))^{\ast}\subset (\mathring{\mV}^1_{-\beta}(\Om))^{\ast}.
\]
Define the operators $A^{\pm\beta}_{\eps}:\mathring{\mV}^1_{\pm\beta}(\Om)\to(\mathring{\mV}^1_{\mp\beta}(\Om))^{\ast}$ such that
\begin{equation}\label{DefOperatorWeight}
\langle A^{\pm\beta}_{\eps}\varphi,\varphi'\rangle=\int_{\Om}\eps\nabla\varphi\cdot\nabla\overline{\varphi'}\,dx,\qquad \forall \varphi\in\mathring{\mV}^1_{\pm\beta}(\Om),\,\varphi'\in\mathring{\mV}^1_{\mp\beta}(\Om).
\end{equation}
Working as in \cite{BonCheCla13} for the 2D case of the corner, one can show that there is $\beta_0>0$ (depending only on $\kappa_{\eps}$ and $\varpi$) such that for all $\beta\in(0;\beta_0)$, $A^{\beta}_{\eps}$ is Fredholm of index $+1$ while $A^{-\beta}_{\eps}$ is Fredholm of index $-1$. We remind the reader that for a bounded linear operator between two Banach spaces $T:\mrm{X}\to\mrm{Y}$ whose range is closed, its index is defined as $\mrm{ind}\,T:=\dim\ker\,T-\dim\coker\,T$, with $\dim\coker\,T=\dim\,(\mrm{Y}/\mrm{range}(T))$. On the other hand, application of Kondratiev calculus guarantees that if $\varphi\in\mathring{\mV}^1_{\beta}(\Om)$ is such that $A^{+\beta}_{\eps}\varphi\in(\mathring{\mV}^1_{\beta}(\Om))^{\ast}$ (the important point here being that $(\mathring{\mV}^1_{\beta}(\Om))^{\ast}\subset (\mathring{\mV}^1_{-\beta}(\Om))^{\ast}$), then there holds the following representation 
\begin{equation}\label{PropoRepresentation}
\varphi=c_-\,s^-+c_+\,s^++\tilde{\varphi}\qquad\mbox{ with }c_{\pm}\in\Cplx\mbox{ and }\tilde{\varphi}\in\mathring{\mV}^1_{-\beta}(\Om).
\end{equation}
Note that  $s^\pm$, with $s^\pm$ defined by (\ref{DefSinguEss}), belongs to $\mathring{\mV}^1_{\beta}(\Om)$, but not to $\mH^1_0(\Omega)$, and a fortiori not to  $\mathring{\mV}^1_{-\beta}(\Om)$. Then introduce the space $\mathring{\mV}^{\mrm{out}}:=\mrm{span}(s^+)\,\oplus\,\mathring{\mV}^1_{-\beta}(\Om)$,  endowed with the norm 
\begin{equation}\label{WPNonStandard}
\|\varphi\|_{\mV^{\mrm{out}}}=(|c|^2+\|\tilde{\varphi}\|^2_{\mV^1_{-\beta}(\Om))})^{1/2},\qquad \forall\varphi=c\,s^++\tilde{\varphi}\in\mathring{\mV}^{\mrm{out}},
\end{equation}
which is a Banach space. Introduce also the operator $A^{\mrm{out}}_{\eps}$ such that for all $\varphi=c\,s^++\tilde{\varphi}\in\mathring{\mV}^{\mrm{out}}$ and $\varphi'\in\mathscr{C}^{\infty}_0(\Om\setminus\{O\})$, 
\[
\langle A^{\mrm{out}}_{\eps}\varphi,\varphi'\rangle=\int_{\Om}\eps\nabla\varphi\cdot\nabla\overline{\varphi'}\,dx=-c\int_{\Om}\div(\eps\nabla s^+)\overline{\varphi'}\,dx+\int_{\Om}\eps\nabla\tilde{\varphi}\cdot\nabla\overline{\varphi'}\,dx.
\]
Note that due to the features of the cut-off function $\chi$, we have $\div(\eps\nabla s^+)\in\mL^2(\Om)$. And since $\div(\eps\nabla s^+)=0$ in a neighbourhood of $O$, we observe that there is a constant $C>0$ such that $|\langle A^{\mrm{out}}_{\eps}\varphi,\varphi'\rangle|\le C\,\|\varphi\|_{\mV^{\mrm{out}}}\,\|\varphi'\|_{\mV^1_{\beta}(\Om)}$. The density of $\mathscr{C}^{\infty}_0(\Om\setminus\{O\})$ in $\mathring{\mV}^1_{\beta}(\Om)$ then allows us to extend $A^{\mrm{out}}_{\eps}$ as a continuous operator from $\mathring{\mV}^{\mrm{out}}$ to $(\mathring{\mV}^1_{\beta}(\Om))^{\ast}$. And we have
\[
\langle A^{\mrm{out}}_{\eps}\varphi,\varphi'\rangle=-c\int_{\Om}\div(\eps\nabla s^+)\overline{\varphi'}\,dx+\int_{\Om}\eps\nabla\tilde{\varphi}\cdot\nabla\overline{\varphi'}\,dx,\quad \forall\varphi=c\,s^++\tilde{\varphi},\,\varphi'\in\mathring{\mV}^1_{\beta}(\Om).
\]
Working as in \cite{BonCheCla13} (see Proposition 4.4.) for the 2D case of the corner, one can prove that $A^{\mrm{out}}_{\eps}:\mathring{\mV}^{\mrm{out}}\to(\mathring{\mV}^1_{\beta}(\Om))^{\ast}$ is Fredholm of index zero and that $\ker\,A^{\mrm{out}}_{\eps}=\ker\,A^{-\beta}_{\eps}$. In order to simplify the analysis below, we shall make the following assumption.\\
\newline
\textbf{Assumption 3.} \textit{We assume that $\eps$ is such that for $\beta\in (0;\beta_0)$,  $A^{-\beta}_{\eps}$ is injective, which guarantees that $A^{\mrm{out}}_{\eps}:\mathring{\mV}^{\mrm{out}}\to(\mathring{\mV}^1_{\beta}(\Om))^{\ast}$ is an isomorphism.}\\
\newline
In what follows, we shall also need to work with the usual Laplace operator  in weighted Sobolev spaces. For $\gamma\in\R$,  define $A^{\gamma}:\mathring{\mV}^1_{\gamma}(\Om)\to(\mathring{\mV}^1_{-\gamma}(\Om))^{\ast}$ such that
\[
\langle A^{\gamma}\varphi,\varphi'\rangle=\int_{\Om}\nabla\varphi\cdot\nabla\overline{\varphi'}\,dx,\qquad \forall \varphi\in\mathring{\mV}^1_{\gamma}(\Om),\,\varphi'\in\mathring{\mV}^1_{-\gamma}(\Om)
\]
(observe that there is no $\eps$ here). Combining the theory presented in \cite{KoMR97} (see also the founding article \cite{Kond67} as well as the monographs \cite{MaNP00,NaPl94}) together with the result of \cite[Corollary 2.2.1]{KoMR01}, we get the following proposition.
\begin{proposition}\label{PropoLaplaceOp}
\noindent For all $\gamma\in(-1/2;1/2)$, the operator $A^{\gamma}:\mathring{\mV}^1_{\gamma}(\Om)\to(\mathring{\mV}^1_{-\gamma}(\Om))^{\ast}$ is an isomorphism.
\end{proposition}
\noindent Note in particular that for $\gamma=0$, this proposition simply says that $\Delta:\mH^1_0(\Om)\to(\mH^1_0(\Om))^{\ast}$ is an isomorphism. In order to have a result of isomorphism both for $A^{\mrm{out}}_{\eps}$ and $A^{\beta}$, we shall often make the assumption that the weight $\beta$  is   such that 
\begin{equation}\label{betainfbeta_0et1demi}
0<\beta<\min(1/2,\beta_0)
\end{equation}
where $\beta_0$ is defined after (\ref{DefOperatorWeight}).
\newline
To measure electromagnetic fields in weighted Sobolev norms, in the following we shall work in the spaces  
\[
\begin{array}{rcl}
\boldsymbol{\mV}^0_{\beta}(\Om)&:=&(\mV^0_{\beta}(\Om))^3\\[4pt]
\boldsymbol{\mathring{\mV}}^1_{\beta}(\Om)&:=&(\mathring{\mV}^1_{\beta}(\Om))^3.
\end{array}
\]
Note that we have $\boldsymbol{\mV}^0_{-\beta}(\Om)\subset \Lspace^2(\Om)\subset \boldsymbol{\mV}^0_{\beta}(\Om)$.

\section{Analysis of the problem for the electric component}\label{SectionChampE}

In this section, we consider the problem for the electric field associated with (\ref{Eqs Maxwell})-(\ref{CL Maxwell}). Since the scalar problem involving $\eps$ is well-posed  in a non standard framework involving the propagating singularity $s^+$ (see (\ref{WPNonStandard})), we shall add its gradient in the space for the electric field. Then we define a variational problem in this unsual space, and prove its well-posedness. Finally we justify our choice by a limiting absorption principle. 
\subsection{A well-chosen space for the electric field}
Define the space of electric fields with the divergence free condition
\begin{equation}
\label{eq-XNout}
\begin{array}{l}
\mX^{\mrm{out}}_N(\eps):=\{\u=c\nabla s^{+}+\tilde{\u},\ c\in\mathbb{C},\,\tilde{\u}\in\Lspace^2(\Om)\,|\,\curl\u\in\Lspace^2(\Omega),\,\div(\eps\u)=0\mbox{ in }\Om\setminus\{O\},\\
\hspace{11.6cm}\u\times\nu=0\mbox{ on }\partial\Om\}.
\end{array}
\end{equation}
In this definition, for $\u=c\nabla s^{+}+\tilde{\u}$, the condition $\div(\eps\u)=0$ in $\Om\setminus\{O\}$ means that there holds 
\begin{equation}\label{defDivFaible}
\int_\Omega \eps\u\cdot\nabla\varphi\,dx=0,\qquad\forall\varphi\in\mathscr{C}^{\infty}_0(\Om\setminus\{O\}),
\end{equation}
which after integration by parts and by density of $\mathscr{C}^{\infty}_0(\Om\setminus\{O\})$ in $\mH^1_0(\Om)$ is equivalent to 
\begin{equation}\label{defDivFaible-bis}
-c\int_{\Om}\div(\eps\nabla s^{+})\varphi\,dx+\int_{\Om}\eps\tilde{\u}\cdot\nabla\varphi\,dx=0,\qquad\forall\varphi\in\mathscr{C}^{\infty}_0(\Om).
\end{equation}
Note that we have $\mX_N(\eps)\subset\mX^{\mrm{out}}_N(\eps)$ and that $\dim\,(\mX^{\mrm{out}}_N(\eps)/\mX_N(\eps))=1$ (see Lemma \ref{LemmaCodim} in Appendix). For $\u=c\nabla s^{+}+\tilde{\u}$ with $c\in\mathbb{C}$ and $\tilde{\u}\in\Lspace^2(\Om)$, we set  
\[
\|\u\|_{\mX^{\mrm{out}}_N(\eps)}=(|c|^2+\|\tilde{\u}\|^2_{\Om}+\|\curl\u\|^2_{\Om})^{1/2}\,.
\]
Endowed with this norm, $\mX^{\mrm{out}}_N(\eps)$ is a Banach space.
\begin{lemma}\label{LemmaNormeEquiv}
Pick some $\beta$ satisfying (\ref{betainfbeta_0et1demi}). Under Assumptions 1 and 3, for any $\u=c\nabla s^{+}+\tilde{\u}\in\mX^{\mrm{out}}_N(\eps)$, we have $\tilde{\u}\in\boldsymbol{\mV}^0_{-\beta}(\Om)$ and there is a constant $C>0$ independent of $\u$ such that 
\begin{equation}\label{EstimBas}
|c|+\|\tilde{\u}\|_{\boldsymbol{\mV}^0_{-\beta}(\Om)} \le C\,\|\curl\u\|_{\Om}.
\end{equation}
As a consequence, the norm $\|\cdot\|_{\mX^{\mrm{out}}_N(\eps)}$ is equivalent to the norm $\|\curl\cdot\|_{\Om}$ in $\mX^{\mrm{out}}_N(\eps)$ and $\mX^{\mrm{out}}_N(\eps)$ endowed with the inner product $(\curl\cdot,\curl\cdot)_{\Om}$ is a Hilbert space.

\end{lemma}
\begin{proof}
Let $\u=c\nabla s^{+}+\tilde{\u}$ be an element of $\mX^{\mrm{out}}_N(\eps)$. The field $\tilde{\u}$ is in $\Lspace^2(\Omega)$ and therefore decomposes as 
\begin{equation}\label{DecompoHelm1}
\tilde{\u}=\nabla\varphi+\curl\boldsymbol{\psi}
\end{equation}
with $\varphi\in\mH^1_{0}(\Om)$ and $\boldsymbol{\psi}\in \mX_T(1)$ (item $iv)$ of Proposition \ref{propoPotential}). Moreover, since $\u\times\nu=0$ on $\partial\Omega$ and  since both $s^+$ and $\varphi$ vanish on $\partial\Omega$, we know that $\curl\boldsymbol{\psi}\times\nu=0$ on $\partial\Omega$. 
Then noting that  $-\boldsymbol{\Delta\psi}=\curl\tilde{\u}=\curl\u\in \Lspace^2(\Omega)$, we deduce  from Proposition \ref{propoLaplacienVect} that $\curl\boldsymbol{\psi} \in \boldsymbol{\mV}^0_{-\beta}(\Om)$ with the estimate
\begin{equation}\label{DecompoHelm2}
\|\curl\boldsymbol{\psi}\|_{\boldsymbol{\mV}^0_{-\beta}(\Om)} \le C\,\|\curl\u\|_{\Om}.
\end{equation}
Using (\ref{defDivFaible}), the condition $\div(\eps\u)=0$ in $\Om\setminus\{O\}$ implies 
\[
\int_{\Om}\eps\nabla(c\,s^++\varphi)\cdot\nabla\varphi'\,dx=-\int_{\Om} \eps \curl\boldsymbol{\psi}\cdot\nabla\varphi'\,dx,\qquad\forall\varphi'\in \mathring{\mV}^1_{-\beta}(\Om),
\]
which means exactly that $A^{\beta}_{\eps}(c\,s^++\varphi)=-\div(\eps\,\curl\boldsymbol{\psi})\in(\mathring{\mV}^1_{-\beta}(\Om))^{\ast}$. Since additionally $-\div(\eps\,\curl\boldsymbol{\psi})\in(\mathring{\mV}^1_{\beta}(\Om))^{\ast}$, from (\ref{PropoRepresentation}) we know that there are some complex constants $c_{\pm}$ and some $\tilde{\varphi}\in\mathring{\mV}^1_{-\beta}(\Om)$ such that
\[
c\,s^++\varphi = c_-\,s^-+c_+\,s^++\tilde{\varphi}.
\]
This implies $c_-=0$, $c_+=c$ (because $\varphi\in\mH^1_0(\Om)$) and so $\varphi=\tilde{\varphi}$ is an element of  $\mathring{\mV}^1_{-\beta}(\Om)$. This shows that $c\,s^++\varphi\in\mathring{\mV}^{\mrm{out}}$ and that $A^{\mrm{out}}_{\eps}(c\,s^++\varphi)=-\div(\eps\,\curl\boldsymbol{\psi}).$ Since $A^{\mrm{out}}_{\eps}:\mathring{\mV}^{\mrm{out}}\to(\mathring{\mV}^1_{\beta}(\Om))^{\ast}$ is an isomorphism, we have the estimate
\begin{equation}\label{DecompoHelm3}
|c|+\|\varphi\|_{\mV^1_{-\beta}(\Om)} \le C\,\|\div(\eps\,\curl\boldsymbol{\psi})\|_{(\mathring{\mV}^1_{\beta}(\Om))^{\ast}} \le C\,\|\curl\boldsymbol{\psi}\|_{\boldsymbol{\mV}^0_{-\beta}(\Om)}.
\end{equation}
Finally gathering (\ref{DecompoHelm1})--(\ref{DecompoHelm3}), we obtain that $\tilde{\u}\in\boldsymbol{\mV}^0_{-\beta}(\Om)$ and that the estimate (\ref{EstimBas}) is valid. Noting that 
$\|\tilde{\u}\|_{\Om}\leq C\,\|\tilde{\u}\|_{\boldsymbol{\mV}^0_{-\beta}(\Om)}$, this implies that the norms $\|\cdot\|_{\mX^{\mrm{out}}_N(\eps)}$ and  $\|\curl\cdot\|_{\Om}$  are equivalent in $\mX^{\mrm{out}}_N(\eps)$.
\end{proof}
\noindent Thanks to the previous lemma and by density of $\mathscr{C}^{\infty}_0(\Om\setminus\{O\})$ in $\mathring{\mV}^1_{\beta}(\Om)$,  the condition (\ref{defDivFaible-bis}) for  $\u=c\nabla s^{+}+\tilde{\u}\in\mX^{\mrm{out}}_N(\eps)$ is equivalent to 
\begin{equation}\label{defDivFaible-ter}
-c\int_{\Om}\div(\eps\nabla s^{+})\varphi\,dx+\int_{\Om}\eps\tilde{\u}\cdot\nabla\varphi\,dx=0,\qquad\forall\varphi\in\mathring{\mV}^1_{\beta}(\Om)
\end{equation}
where all the terms are well-defined as soon as $\beta$ satisfies (\ref{betainfbeta_0et1demi}).

\subsection{Definition of the problem for the electric field}

Our objective is to define the problem for the electric field as a variational formulation set in $\mX^{\mrm{out}}_N(\eps)$. For some $\gamma>0$, let $\J$ be an element of $\boldsymbol{\mV}^0_{-\gamma}(\Om)$ such that $\div\,\J=0$ in $\Om$. Consider the problem  
\begin{equation}\label{MainPbE}
	\begin{array}{|l}
		\mbox{Find }\u\in\mX^{\mrm{out}}_N(\eps)\mbox{ such that }\\
		\dsp\int_{\Om}\mu^{-1}\curl\u\cdot\curl\overline{\v}\,dx -\om^2 \fint_\Om \eps \u\cdot\overline{\v}\,dx
		= i\om\int_\Om \J\cdot\overline{\v}\,dx,\qquad \forall\v\in\mX^{\mrm{out}}_N(\eps),
	\end{array}
\end{equation}
where the term 
\begin{equation}\label{defInt}
\fint_\Om \eps \u\cdot\overline{\v}\,dx
\end{equation}
has to be carefully defined. The difficulty comes from the fact that $\mX^{\mrm{out}}_N(\eps)$ is not a subspace of $\Lspace^2(\Om)$ so that this quantity cannot be considered as a classical integral. \\
Let $\u=c_{\u}\nabla s^{+}+\tilde{\u}\in \mX^{\mrm{out}}_N(\eps)$. First, for $\tilde{\v}\in \boldsymbol{\mV}^0_{-\beta}(\Om)$ with $\beta>0$, it is natural to set 
\begin{equation}
	\label{def-newint-vtilde}\fint_\Om \eps \u\cdot\overline{\tilde{\v}}\,dx:=\int_\Om \eps \u\cdot\overline{\tilde{\v}}\,dx.
\end{equation}
To complete the definition, we have to give a sense to (\ref{defInt}) when $\v=\nabla s^+$. Proceeding as for the derivation of (\ref{defDivFaible-ter}), we start from the identity
\[
\int_\Om \eps \u\cdot\overline{\nabla\varphi}\,dx=-c_{\u}\int_{\Om}\div(\eps\nabla s^{+})\overline{\varphi} \,dx+\int_{\Om}\eps\tilde{\u}\cdot\overline{\nabla\varphi}\,dx,\qquad\forall\varphi\in\mathscr{C}^{\infty}_0(\Om\setminus\{O\}).
\]
By density of $\mathscr{C}^{\infty}_0(\Om\setminus\{O\})$ in $\mathring{\mV}^1_{\beta}(\Om)$, this leads to set 
\begin{equation}\label{def-newint-V1beta}
	\fint_\Om \eps \u\cdot \overline{\nabla\varphi}\,dx:=-c_{\u}\int_{\Om}\div(\eps\nabla s^{+})\overline{\varphi} \,dx+\int_{\Om}\eps\tilde{\u}\cdot\overline{\nabla\varphi}\,dx,\qquad\forall\varphi\in\mathring{\mV}^1_{\beta}(\Om).
\end{equation} 
With this definition, condition \eqref{defDivFaible-ter} can be written as 
$$
	\fint_\Om \eps \u\cdot \nabla \varphi\,dx=0,\qquad\forall\varphi\in\mathring{\mV}^1_{\beta}(\Om).
$$
In particular, since $s^+\in \mathring{\mV}^1_{\beta}(\Om)$, for all $\u\in\mX^{\mrm{out}}_N(\eps)$ we have
\begin{equation}\label{defDivFaible-4}
\fint_\Om \eps \u\cdot \nabla \overline{s^+}\,dx=0\qquad\mbox{and so}\qquad \int_\Omega \eps \tilde{\u}\cdot\nabla \overline{s^{+}}\,dx=c_{\u} \int_{\Om}\div(\eps\nabla{s^+})\overline{s^{+}}\,dx.
\end{equation}
Finally for all $\u=c_{\u}\nabla s^{+}+\tilde{\u}$ and  $\v=c_{\v}\nabla s^{+}+\tilde{\v}$ in $\mX^{\mrm{out}}_N(\eps)$, using \eqref{def-newint-vtilde} and \eqref{defDivFaible-4}, we find 
\[
\fint_\Om \eps \u\cdot\overline{\v}\,dx=\int_\Om \eps \u\cdot\overline{\tilde{\v}}\,dx=c_{\u}\int_\Omega \eps \nabla s^{+}\cdot\overline{\tilde{\v}}\,dx+\int_\Om \eps \tilde{\u}\cdot\overline{\tilde{\v}}\,dx.
\]
But since $\v\in\mX^{\mrm{out}}_N(\eps)$, we deduce from the second identity of \eqref{defDivFaible-4} that
\begin{equation}\label{RelationPart}
\int_\Omega \eps \nabla s^+\cdot\overline{\tilde{\v}}\,dx=\overline{c_{\v}} \int_{\Om}\div(\eps\nabla{\overline{s^{+}} )}{s^+}\,dx.
\end{equation}
Summing up, we get 
\begin{equation}\label{newint-epsuvbar}
	\fint_\Om \eps \u\cdot\overline{\v}\,dx=c_{\u}\overline{c_{\v}}\int_{\Om}\div(\eps\nabla{\overline{s^{+}} )}{s^+}\,dx+\int_\Om \eps \tilde{\u}\cdot\overline{\tilde{\v}}\,dx,\qquad \forall\u, \v\in\mX^{\mrm{out}}_N(\eps).
\end{equation}
\begin{remark}
Even if we use an integral symbol to keep the usual aspects of formulas and facilitate the reading, it is important to consider this new quantity as a  sesquilinear form 
$$(\u, \v) \mapsto 	\fint_\Om \eps \u\cdot\overline{\v}\,dx$$
on $\mX^{\mrm{out}}_N(\eps)\times\mX^{\mrm{out}}_N(\eps)$. In particular, we point out that this sesquilinear form 
is not hermitian on $\mX^{\mrm{out}}_N(\eps)\times\mX^{\mrm{out}}_N(\eps)$. Indeed, we have
\[
\overline{\fint_\Om \eps \v\cdot\overline{\u}\,dx}=\int_\Om \eps \tilde{\u}\cdot\overline{\tilde{\v}}\,dx+c_{\u}\overline{c_{\v}}\int_{\Om}\div(\eps\nabla{s^{+} )}{\overline{s^+}}\,dx
\]
so that 
\begin{equation}\label{FormulaAdjoint}
\fint_\Om \eps \u\cdot\overline{\v}\,dx-\overline{\fint_\Om \eps \v\cdot\overline{\u}\,dx}=2ic_{\u}\overline{c_{\v}}\,\Im m\,\bigg(\int_{\Om}\div(\eps\nabla\overline{s^{+}} )\,s^+\,dx\bigg).
\end{equation}
But 
Lemma \ref{lemmaNRJ} and assumption \eqref{eq-intPhineq0} show that
\[
\Im m\,\bigg(\int_{\Om}\div(\eps\nabla\overline{s^{+}} )\,s^+\,dx\bigg) \neq 0.
\]
\end{remark}
\noindent In the sequel, we denote by $a_N(\cdot,\cdot)$ (resp. $\ell_N(\cdot)$) the sesquilinear form (resp. the antilinear form) appearing in the left-hand side (resp. right-hand side) of (\ref{MainPbE}).

\subsection{Equivalent formulation}
Define the space 
\[
\Hspace_N^{\mrm{out}}(\curl):=\mrm{span}(\nabla s^+)\oplus\Hspace_N(\curl)\supset\mX^{\mrm{out}}_N(\eps)
\]
(without the divergence free condition) and consider the problem 
\begin{equation}\label{Formu2}
\begin{array}{|l}
\mbox{Find }\u\in \Hspace_N^{\mrm{out}}(\curl)\mbox{ such that }\\[2pt]
a_N(\u,\v)= \ell_N(\v),\quad\forall \v \in \Hspace_N^{\mrm{out}}(\curl),
\end{array}
\end{equation}
where  the definition of 
	$$\fint_\Om \eps \u\cdot\overline{\v}\,dx$$
	has to be extended to the space $\Hspace_N^{\mrm{out}}(\curl)$. Working exactly as in the beginning of the proof of Lemma \ref{LemmaNormeEquiv}, one can show that any $\u\in\Hspace_N^{\mrm{out}}(\curl)$ admits the decomposition 
\begin{equation}\label{DecompoGene}
\u=c_{\u}\nabla s^++\nabla\varphi_{\u}+\curl\boldsymbol{\psi}_{\u},
\end{equation}
with $c_{\u}\in\Cplx$, $\varphi_{\u}\in\mH^1_{0}(\Om)$ and $\boldsymbol{\psi}_{\u}\in \mX_T(1)$, such that $\curl\boldsymbol{\psi}_{\u}\in \boldsymbol{\mV}^0_{-\beta}(\Om)$, for $\beta$ satisfying (\ref{betainfbeta_0et1demi}). Then, for all $\u=c_{\u}\nabla s^++\nabla\varphi_{\u}+\curl\boldsymbol{\psi}_{\u}$ and $\v=c_{\v}\nabla s^++\nabla\varphi_{\v}+\curl\boldsymbol{\psi}_{\v}$ in $\Hspace_N^{\mrm{out}}(\curl)$, a natural extension of the previous definitions  leads to set
\begin{equation}\label{DefExtensionH}
\begin{array}{lcl}
\fint_{\Om}\eps\u\cdot\overline{\v}\,dx&:=&\int_{\Om}\eps\,(\nabla\varphi_{\u}+\curl\boldsymbol{\psi}_{\u})\cdot(\nabla\overline{\varphi_{\v}}+\curl\overline{\boldsymbol{\psi}_{\v}})\,dx\\[10pt]
& & +\int_{\Om}c_{\u}\,\eps\nabla s^+\cdot\overline{\curl\boldsymbol{\psi}_{\v}}+\overline{c_{\v}}\,\eps\,\curl\boldsymbol{\psi}_{\u}\cdot\nabla \overline{s^+}\,dx\\[10pt]
& & -\int_{\Om}c_{\u}\overline{c_{\v}}\,\div(\eps\nabla s^+)\overline{s^+}+c_{\u}\,\div(\eps\nabla s^+)\overline{\varphi_{\v}}+\overline{c_{\v}}\,\varphi_{\u}\div(\eps\nabla \overline{s^+})\,dx.
\end{array}
\end{equation}
Note that (\ref{DefExtensionH}) is indeed an extension of 
(\ref{newint-epsuvbar}). To show it, first observe that for $\u=c_{\u}\nabla s^++\nabla\varphi_{\u}+\curl\boldsymbol{\psi}_{\u}$, $\v=c_{\v}\nabla s^++\nabla\varphi_{\v}+\curl\boldsymbol{\psi}_{\v}$ in $\mX^{\mrm{out}}_N(\eps)$, the proof of Lemma \ref{LemmaNormeEquiv} guarantees that $\varphi_{\u}$, $\varphi_{\v}\in\mathring{\mV}^1_{-\beta}(\Om)$ with $\beta$ satisfying (\ref{betainfbeta_0et1demi}). This allows us to integrate by parts in the last two terms of (\ref{DefExtensionH}) to get
\begin{equation}\label{Calcul1}
\begin{array}{lcl}
\fint_{\Om}\eps\u\cdot\overline{\v}\,dx&:=&\int_{\Om}\eps\,(\nabla\varphi_{\u}+\curl\boldsymbol{\psi}_{\u})\cdot(\nabla\overline{\varphi_{\v}}+\curl\overline{\boldsymbol{\psi}_{\v}})\,dx\\[10pt]
& & +\int_{\Om}c_{\u}\,\eps\nabla s^+\cdot(\nabla\overline{\varphi_{\v}}+\overline{\curl\boldsymbol{\psi}_{\v}})+\overline{c_{\v}}\,\eps\,(\nabla\varphi_{\u}+\curl\boldsymbol{\psi}_{\u})\cdot\nabla \overline{s^+}\,dx\\[10pt]
& & -c_{\u}\overline{c_{\v}}\int_{\Om}\div(\eps\nabla s^+)\overline{s^+}\,dx.
\end{array}
\end{equation}
Using (\ref{defDivFaible-4}), (\ref{RelationPart}), the second line above can be written as 
\begin{equation}\label{Calcul2}
\begin{array}{ll}
&\int_{\Om}c_{\u}\,\eps\nabla s^+\cdot(\nabla\overline{\varphi_{\v}}+\overline{\curl\boldsymbol{\psi}_{\v}})+\overline{c_{\v}}\,\eps\,(\nabla\varphi_{\u}+\curl\boldsymbol{\psi}_{\u})\cdot\nabla \overline{s^+}\,dx\\[10pt]
=&c_{\u}\overline{c_{\v}}\int_{\Om}\div(\eps\nabla \overline{s^+})s^+\,dx+c_{\u}\overline{c_{\v}}\int_{\Om}\div(\eps\nabla s^+)\overline{s^+}\,dx.
\end{array}
\end{equation}
Inserting (\ref{Calcul2}) in (\ref{Calcul1}) yields exactly (\ref{newint-epsuvbar}).
\begin{lemma}\label{lemmaEquivE}
Under Assumptions 1 and 3,  the field $\u$ is a solution of (\ref{MainPbE}) if and only if it solves the problem (\ref{Formu2}).
\end{lemma}
\begin{proof}
If $\u\in\Hspace_N^{\mrm{out}}(\curl)$ satisfies (\ref{Formu2}), then taking $\v=\nabla\varphi$ with $\varphi\in\mathscr{C}^{\infty}_0(\Om\setminus\{O\})$ in (\ref{Formu2}), and using that $\div\,\J=0$ in $\Om$, we get \eqref{defDivFaible}, which implies that $\u\in\mX^{\mrm{out}}_N(\eps)$. This shows that $\u$ solves (\ref{MainPbE}).\\
\newline
Now assume that $\u\in\mX^{\mrm{out}}_N(\eps)\subset\Hspace_N^{\mrm{out}}(\curl)$ is a solution of (\ref{MainPbE}). Let $\v$ be an element of $\Hspace_N^{\mrm{out}}(\curl)$. As in (\ref{DecompoGene}), we have the decomposition 
\begin{equation}\label{DecompoGenev}
\v=c_{\v}\nabla s^++\nabla\varphi_{\v}+\curl\boldsymbol{\psi}_{\v},
\end{equation}
with $c_{\v}\in\Cplx$, $\varphi_{\v}\in\mH^1_{0}(\Om)$ and $\boldsymbol{\psi}_{\v}\in \mX_T(1)$ such that $\curl\boldsymbol{\psi}_{\v}\in \boldsymbol{\mV}^0_{-\beta}(\Om)$ for all $\beta$ satisfying (\ref{betainfbeta_0et1demi}). By Assumption 3, there is $\zeta \in \mathring{\mV}^{\mrm{out}}$ such that 
\begin{equation}\label{ResolPbAdd}
A^{\mrm{out}}_{\eps}\zeta=-\div(\eps\,\curl\boldsymbol{\psi}_{\v}) \in(\mathring{\mV}^1_{\beta}(\Om))^{\ast} .
\end{equation}
The function $\zeta$ decomposes as $\zeta=\alpha s^++\tilde{\zeta}$ with $\tilde{\zeta}\in\mathring{\mV}^1_{-\beta}(\Om)$. Finally, set 
\[
\hat{\v}=\curl\boldsymbol{\psi}_{\v}-\nabla \zeta=\v-\nabla(c_{\v}s^++\varphi_{\v}+\zeta).
\]
The function $\hat{\v}$ is in $\mX^{\mrm{out}}_N(\eps)$, it satisfies $\curl\hat{\v}=\curl\v$  and  from \eqref{defDivFaible-4}, 
we deduce that
\[
\fint_{\Om}\eps\u\cdot\overline{\hat{\v}}\,dx=\fint_{\Om}\eps\u\cdot\overline{\v}\,dx.
\]
Using also that $\J\in\boldsymbol{\mV}^0_{-\gamma}(\Om)$ for some $\gamma>0$ and is such that $\div\,\J=0$ in $\Om$, so that
\[
\int_\Om \J\cdot\overline{\hat{\v}}\,dx=\int_\Om \J\cdot\overline{\v}\,dx,
\]
this shows that $a_N(\u,\v)=a_N(\u,\hat{\v})=\ell_N(\hat{\v})=\ell_N(\v)$ and ends the proof. 
\end{proof}
\noindent In the following, we shall work with the formulation (\ref{MainPbE}) set in $\mX^{\mrm{out}}_N(\eps)$. The reason being that, as usual in the analysis of Maxwell's equations, the divergence free condition will yield a compactness property allowing us to deal with the term involving the frequency $\om$. 
\subsection{Main analysis for the electric field}\label{subsec-mainanalysisE}
\label{subsec-mainanalysisE}
\noindent Define the continuous operators $\mathbb{A}^{\mrm{out}}_N:\mX^{\mrm{out}}_N(\eps)\to(\mX^{\mrm{out}}_N(\eps))^{\ast}$ and $\mathbb{K}^{\mrm{out}}_N:\mX^{\mrm{out}}_N(\eps)\to(\mX^{\mrm{out}}_N(\eps))^{\ast}$ such that for all $\u,\,\v\in\mX^{\mrm{out}}_N(\eps)$, 
\[
\langle\mathbb{A}^{\mrm{out}}_N\u,\v\rangle = \int_{\Om}\mu^{-1}\curl\u\cdot\curl\overline{\v}\,dx,\qquad\qquad\langle\mathbb{K}^{\mrm{out}}_N\u,\v\rangle = \fint_{\Om}\eps\u\cdot\overline{\v}\,dx.
\]
With this notation, we have $\langle(\mathbb{A}^{\mrm{out}}_N+\mathbb{K}^{\mrm{out}}_N)\u,\v\rangle=a_{N}(\u,\v)$. 

\begin{proposition}\label{propoUniqueness}
Under Assumptions 1--3, the operator $\mathbb{A}^{\mrm{out}}_N:\mX^{\mrm{out}}_N(\eps)\to(\mX^{\mrm{out}}_N(\eps))^{\ast}$ is an isomorphism. 
\end{proposition}
\begin{proof}
Let us construct a continuous operator $\mathbb{T}:\mX^{\mrm{out}}_N(\eps)\to\mX^{\mrm{out}}_N(\eps)$ such that for all $\u,\,\v\in\mX^{\mrm{out}}_N(\eps)$, 
\[
\int_{\Om}\mu^{-1}\curl\u\cdot\curl(\overline{\mathbb{T}\v})\,dx=\int_{\Om}\curl\u\cdot\curl\overline{\v}\,dx.
\]
To proceed, we adapt the method presented in \cite{BoCC14}. Assume that $\v\in\mX^{\mrm{out}}_N(\eps)$ is given. We construct $\mathbb{T}\v$ in three steps.\\
\newline
1) Since $\curl\v\in\Lspace^2(\Omega)$ and $A_{\mu}:\mH^1_{\#}(\Om)\to(\mH^1_{\#}(\Om))^{\ast}$ is an isomorphism, there is a unique $\zeta\in\mH^1_{\#}(\Om)$ such that 
\[
\int_{\Om}\mu\nabla\zeta\cdot\nabla\zeta'\,dx=\int_{\Om}\mu\,\curl\v\cdot\nabla\zeta'\,dx,\qquad\forall\zeta'\in\mH^1_{\#}(\Om).
\]
Then the field  $\mu(\curl\v-\nabla\zeta)\in\Lspace^2(\Omega) $ is divergence free in $\Om$ and satisfies $\mu(\curl\v-\nabla\zeta)\cdot\nu=0$ on $\partial\Om$.\\
\newline
2) From item $ii)$ of Proposition \ref{propoPotential}, we infer that there is $\boldsymbol{\psi}\in \mX_N(1)$ such that
$$
\mu(\curl\v-\nabla\zeta)= \curl \boldsymbol{\psi}.
$$
Thanks to Lemma \ref{LemmaWeightedClaBis}, we deduce that $\boldsymbol{\psi}\in \boldsymbol{\mV}^0_{-\beta}(\Om)$ for all $\beta\in(0;1/2)$ and a fortiori  for  $\beta$ satisfying (\ref{betainfbeta_0et1demi}).\\
\newline
3) Suppose now that   $\beta$ satisfies (\ref{betainfbeta_0et1demi}). Then we know from the previous step that  $\div(\eps \boldsymbol{\psi})\in(\mathring{\mV}^1_{\beta}(\Om))^{\ast}$. On the other hand, by Assumption 3,  $A^{\mrm{out}}_{\eps}:\mathring{\mV}^{\mrm{out}}\to(\mathring{\mV}^1_{\beta}(\Om))^{\ast}$ is an isomorphism. Consequently we can introduce $\varphi\in \mathring{\mV}^{\mrm{out}}$ such that $A^{\mrm{out}}_{\eps}\varphi=-\div(\eps \boldsymbol{\psi})$.\\ 
\newline
Finally, we set $\mathbb{T}\v=\boldsymbol{\psi}-\nabla \varphi$. Clearly $\mathbb{T}\v$ is an element of $\mX^{\mrm{out}}_N(\eps)$. Moreover, for all $\u$, $\v$ in $\mX^{\mrm{out}}_N(\eps)$, we have 
\[
\begin{array}{lcl}
\dsp\int_\Om \mu^{-1} \curl\u\cdot\curl\overline{\mathbb{T}\v}\,dx&=&\dsp\int_\Om \mu^{-1} \curl\u\cdot\curl\overline{\boldsymbol{\psi}}\,dx\\[10pt]
&=&\dsp\int_\Om  \curl\u\cdot\curl\overline{\v}\,dx-\dsp\int_\Om  \curl\u\cdot\nabla\overline{\zeta}\,dx\\[10pt]
&=&\dsp\int_\Om  \curl\u\cdot\curl\overline{\v}\,dx.
\end{array}
\]
From Lemma \ref{LemmaNormeEquiv} and the Lax-Milgram theorem, we deduce that $\mathbb{T}^{\ast}\mathbb{A}^{\mrm{out}}_N:\mX^{\mrm{out}}_N(\eps)\to(\mX^{\mrm{out}}_N(\eps))^{\ast}$ is an isomorphism. And by symmetry, permuting the roles of $\u$ and $\v$, it is obvious that $\mathbb{T}^{\ast}\mathbb{A}^{\mrm{out}}_N=\mathbb{A}^{\mrm{out}}_N\mathbb{T}$, which allows us to conclude that $\mathbb{A}^{\mrm{out}}_N:\mX^{\mrm{out}}_N(\eps)\to(\mX^{\mrm{out}}_N(\eps))^{\ast}$ is an isomorphism.
\end{proof}

\begin{proposition}\label{PropCompactE}
Under Assumptions 1 and 3, if $(\u_n=c_n\nabla s^++\tilde{\u}_n)$ is a sequence which is bounded in $\mX^{\mrm{out}}_N(\eps)$, then we can extract a subsequence such that $(c_n)$ and $(\tilde{\u}_n)$ converge respectively in $\mathbb{C}$ and in $\boldsymbol{\mV}^0_{-\beta}(\Om)$ for $\beta$ satisfying (\ref{betainfbeta_0et1demi}). As a consequence, the operator $\mathbb{K}^{\mrm{out}}_N:\mX^{\mrm{out}}_N(\eps)\to(\mX^{\mrm{out}}_N(\eps))^{\ast}$ is compact.
\end{proposition}
\begin{proof} 
	
Let $(\u_n)$ be a bounded sequence of elements of $\mX^{\mrm{out}}_N(\eps)$. From the proof of Lemma \ref{LemmaNormeEquiv}, we know that for $n\in\N$, we have 
\begin{equation}\label{DecompoCompa}
\u_n=c_n\nabla s^++\nabla\varphi_n+\curl\boldsymbol{\psi}_n
\end{equation}
where the sequences $(c_n)$, $(\varphi_n)$, $(\boldsymbol{\psi}_n)$ and $(\curl\boldsymbol{\psi}_n)$ are bounded respectively in $\mathbb{C}$, $\mathring{\mV}^1_{-\beta}(\Om)$, $\mX_T(1)$ and $\boldsymbol{\mV}^0_{-\beta}(\Om)$. Observing that $\curl\u_n=\curl\curl\boldsymbol{\psi_n}=-\boldsymbol{\Delta \psi_n} $ is bounded in $\Lspace^2(\Omega)$, we deduce from Proposition \ref{propoLaplacienVectcompact} that there exists a subsequence such that  $\curl\boldsymbol{\psi_n}$ converges in $\boldsymbol{\mV}^0_{-\beta}(\Om)$.  Moreover, by  \eqref{DecompoHelm3}, we have
\[
|c_n-c_m|+\|\varphi_n-\varphi_m\|_{\mV^1_{-\beta}(\Om)}\leq C\|\curl(\boldsymbol{\psi}_n-\boldsymbol{\psi}_m)\|_{\boldsymbol{\mV}^0_{-\beta}(\Om)},
\]
which implies that $(c_n)$ and $(\varphi_n)$ converge respectively in $\mathbb{C}$ and in $\mathring{\mV}^1_{-\beta}(\Om)$. From (\ref{DecompoCompa}), we see that this is enough to conclude to the first part of the proposition.\\
Finally, observing that
\[
\|\mathbb{K}_N^{\mrm{out}}\u\|_{(\mX^{\mrm{out}}_N(\eps))^{\ast}} \le C\,(\|\tilde{\u}\|_{\boldsymbol{\mV}^0_{-\beta}(\Om)}+|c_{\u}|),
\]
we deduce that $\mathbb{K}^{\mrm{out}}_N:\mX^{\mrm{out}}_N(\eps)\to(\mX^{\mrm{out}}_N(\eps))^{\ast}$ is a compact operator.
\end{proof}
\noindent We can now state the main theorem of the analysis of the problem for the electric field.
\begin{theorem}\label{MainThmE}
Under Assumptions 1--3, for all $\om\in\R$ the operator $\mathbb{A}^{\mrm{out}}_N-\om^2\mathbb{K}^{\mrm{out}}_N:\mX^{\mrm{out}}_N(\eps)\to(\mX^{\mrm{out}}_N(\eps))^{\ast}$ is Fredholm of index zero. 
\end{theorem}
\begin{proof}
Since $\mathbb{K}^{\mrm{out}}_N:\mX^{\mrm{out}}_N(\eps)\to(\mX^{\mrm{out}}_N(\eps))^{\ast}$ is compact (Proposition \ref{PropCompactE}) and $\mathbb{A}^{\mrm{out}}_N:\mX^{\mrm{out}}_N(\eps)\to(\mX^{\mrm{out}}_N(\eps))^{\ast}$ is an isomorphism (Proposition \ref{propoUniqueness}),  $\mathbb{A}^{\mrm{out}}_N-\om^2\mathbb{K}^{\mrm{out}}_N:\mX^{\mrm{out}}_N(\eps)\to(\mX^{\mrm{out}}_N(\eps))^{\ast}$ is Fredholm of index zero. 
\end{proof}
\noindent The previous theorem guarantees that the problem (\ref{MainPbE}) is well-posed if and only if uniqueness holds, that is if and only if the only solution for $\J=0$ is $\u=0$. Since uniqueness holds for $\om=0$, one can prove with the analytic Fredholm theorem that (\ref{MainPbE}) is well-posed except for at most a countable set of values of $\om$ with no accumulation points (note that Theorem \ref{MainThmE} remains true for $\om\in\Cplx$). \\
However this result is not really relevant from a physical point of view. Indeed, negative values of $\eps$ can occur only if $\eps$ is itself a function of $\om$. For instance,  if the inclusion  $\mathcal{M}$ is metallic,  it is commonly admitted that the Drude's law gives a good model for $\eps$. But taking into account the dependence of $\eps$ with respect to $\om$ when studying uniqueness of problem (\ref{MainPbE}) leads to a non-linear eigenvalue problem, where the functional space $\mX^{\mrm{out}}_N(\eps)$ itself depends on $\om$. This study is beyond the scope of the present paper (see \cite{HaPa17} for such questions in the case of the 2D scalar problem). \\
Nonetheless, there is a result that we can prove concerning the cases of non-uniqueness for problem (\ref{MainPbE}). 
\begin{proposition}\label{trappedmodes}
If $\u=c\nabla s^{+}+\tilde{\u}\in\mX^{\mrm{out}}_N(\eps)$ is a solution of (\ref{MainPbE}) for $\J=0$, then $c=0$ and $\u\in\mX_N(\eps)$.  
\end{proposition}
\begin{proof}
When $\om=0$, the result is a direct consequence of Theorem \ref{MainThmE} (because zero is the only solution of (\ref{MainPbE}) for $\J=0$). From now on, we assume that $\om\in\R\setminus\{0\}$. Suppose that  $\u=c\nabla s^{+}+\tilde{\u}\in\mX^{\mrm{out}}_N(\eps)$ is such that
\[
\int_{\Om}\mu^{-1}\curl\u\cdot\curl\overline{\v}\,dx -\om^2 \fint_\Om \eps \u\cdot\overline{\v}\,dx
	= 0,\qquad \forall\v\in\mX^{\mrm{out}}_N(\eps).
\]	
Taking the imaginary part of the previous identity for $\v=\u$, we get
\[
\Im m\,\bigg(\fint_\Om \eps \u\cdot\overline{\u}\,dx\bigg) =0.
\]
On the other  hand, by (\ref{newint-epsuvbar}), we have
\[
\fint_\Om \eps \u\cdot\overline{\u}\,dx=\int_\Om \eps |\tilde{\u}|^2\,dx+|c|^2\int_{\Om}\div(\eps\nabla\overline{s^{+}} )\,{s^+}\,dx,
\]
so that
\[
|c|^2	\Im m\,\bigg(\int_{\Om}\div(\eps\nabla\overline{s^{+}} )\,s^+\,dx\bigg) =0.
\]
The result of the proposition is then a consequence of Lemma \ref{lemmaNRJ} in Appendix where it is proved that
\[
\Im m\,\bigg(\int_{\Om}\div(\eps\nabla\overline{s^{+}} )\,s^+\,dx\bigg) =\eta\int_{\mathbb{S}^2}\eps|\Phi|^2ds,
\]
and of the assumption \eqref{eq-intPhineq0}.
\end{proof}
\begin{remark}As a consequence, from Lemma \ref{LemmaNormeEquiv}, we infer that elements of the kernel of $\mathbb{A}^{\mrm{out}}_N-\om^2\mathbb{K}^{\mrm{out}}_N$ are in $\boldsymbol{\mV}^0_{-\beta}(\Om)$ for all $\beta$ satisfying  (\ref{betainfbeta_0et1demi}).
\end{remark}

\subsection{Problem in the classical framework}

In the previous paragraph, we have shown that the Maxwell's problem (\ref{MainPbE}) for the electric field set in the non standard space $\mX^{\mrm{out}}_N(\eps)$, and so in $\Hspace_N^{\mrm{out}}(\curl)$ according to Lemma \ref{lemmaEquivE}, is well-posed. Here, we wish to analyse the properties of the problem for the electric field set in the classical space $\mX_N(\eps)$ (which does not contain $\nabla s^+$). Since this space is a closed subspace of $\mX^{\mrm{out}}_N(\eps)$, it inherits the main properties of the problem in $\mX^{\mrm{out}}_N(\eps)$ proved in the previous section. More precisely, we deduce from Lemma \ref{LemmaNormeEquiv} and Proposition \ref{PropCompactE}  the following result. 
\begin{proposition}\label{PropX_N(eps)}
Under Assumptions 1 and 3, the embedding of $\mX_N(\eps)$ in $\Lspace^2(\Om)$ is compact, and  $\|\curl\cdot\|_{\Om}$ is a norm in $\mX_N(\eps)$ which is equivalent to the norm  $\|\cdot\|_{\Hspace(\curl)}$.
\end{proposition}
\noindent Note that we recover classical properties similar to what is known for positive $\eps$, or more generally \cite{BoCC14} for $\eps$ such that the operator $A_{\eps}:\mH^1_0(\Om)\to(\mH^1_0(\Om))^{\ast}$ defined by (\ref{DefAeps}) is an isomorphism (which allows for sign-changing $\eps$). But we want to underline  the fact that under Assumption 3, these classical results could not be proved by using   classical arguments.  
They require  the introduction of the bigger space $\mX^{\mrm{out}}_N(\eps)$, with  the singular function $\nabla s^+$.
\newline 
Let us now consider the problem
\begin{equation}\label{MainPbECla}
\begin{array}{|l}
\mbox{Find }\u\in\mX_N(\eps)\mbox{ such that }\\
\dsp\int_{\Om}\mu^{-1}\curl\u\cdot\curl\overline{\v}\,dx -\om^2 \int_\Om \eps \u\cdot\overline{\v}\,dx= i\om\int_\Om \J\cdot \overline{\v}\,dx,\qquad \forall\v\in\mX_N(\eps).
\end{array}
\end{equation}
An important remark is that one cannot prove that problem (\ref{MainPbECla}) is equivalent to a similar problem set in $\Hspace_N(\curl)$ (the analogue of Lemma \ref{lemmaEquivE}). Again, the difficulty comes from the fact that $A_{\eps}$ is not an isomorphism, and the trouble would appear when solving  (\ref{ResolPbAdd}). Therefore, a solution of (\ref{MainPbECla}) is not in general a distributional solution of the equation 
\[
\curl \left(\mu^{-1}\curl \u\right)-\omega^2\eps u=i\om \J. 
\]
To go further in the analysis of (\ref{MainPbECla}), we recall that  $\mX_N(\eps)$ is a subspace of codimension one of $\mX^{\mrm{out}}_N(\eps)$ (Lemma \ref{LemmaCodim} in Appendix). Let $\v_0$ be an element of $\mX^{\mrm{out}}_N(\eps)$ which does not belong to $\mX_N(\eps)$. Then we denote by $\ell_0$ the continuous linear form on $\mX^{\mrm{out}}_N(\eps)$ such that:
\begin{equation}
\label{eq-decompv_0}
\forall v\in\mX^{\mrm{out}}_N(\eps)\qquad\v-\ell_0(\v)\v_0\in \mX_N(\eps).
\end{equation}
Let us now define the operators $\mathbb{A}_N:{\mX}_N(\eps)\to(\mX_N(\eps))^{\ast}$ and $\mathbb{K}_N:{\mX}_N(\eps)\to(\mX_N(\eps))^{\ast}$ by
\[
\langle\mathbb{A}_N\u,\v\rangle = \int_{\Om}\mu^{-1}\curl\u\cdot\curl\overline{\v}\,dx,\qquad \langle\mathbb{K}_N\u,\v\rangle = \int_{\Om}\eps\u\cdot\overline{\v}\,dx.
\]
\begin{proposition}\label{propoFredholmclas}
	Under Assumptions 1--3, the operator $\mathbb{A}_N:{\mX}_N(\eps)\to(\mX_N(\eps))^{\ast}$ is Fredholm of index zero.
\end{proposition}
\begin{proof}
Let $\u\in \mX_N(\eps)$. By Proposition \ref{propoUniqueness}, for the operator ${\mathbb T}$ introduced in the corresponding proof, one has:
$$\|\u\|^2_{\mX_N(\eps)}=\|\curl \u\|^2_\Om=\langle\mathbb{A}^{\mrm{out}}_N\u,{\mathbb T}\u\rangle.$$
Then, using (\ref{eq-decompv_0}), we get:
$$\|\u\|^2_{\mX_N(\eps)}=\langle\mathbb{A}_N\u,{\mathbb T}\u-\ell_0({\mathbb T}\u)v_0\rangle+\langle\mathbb{A}^{\mrm{out}}_N\u,\ell_0({\mathbb T}\u)v_0\rangle,$$
which implies that
$$\|\u\|_{\mX_N(\eps)}\leq C\left(\|\mathbb{A}_N\u\|_{(\mX_N(\eps))^{\ast}}+|\ell_0({\mathbb T}\u)|\right) .$$
The result of the proposition then follows from a classical adaptation of Peetre's lemma (see for example \cite[Theorem 12.12]{Wlok87}) together with the fact that $\mathbb{A}_N$ is bounded and hermitian.
\end{proof}
\noindent Combining the two previous propositions, we obtain the
\begin{theorem}	\label{the-Fredholm-cla}
Under Assumptions 1--3, for all $\om\in\R$, the operator $\mathbb{A}_N-\om^2\mathbb{K}_N:{\mX}_N(\eps)\to(\mX_N(\eps))^{\ast}$ is Fredholm of index zero.	
\end{theorem}
\noindent But as mentioned above, even if uniqueness holds and if Problem (\ref{MainPbECla}) is well-posed, it does not provide a solution of Maxwell's equations.

\subsection{Expression of the singular coefficient}\label{paraSingCoef}
Under Assumptions 1--3, Theorem \ref{MainThmE} guarantees that for all $\om\in\R$ the operator $\mathbb{A}^{\mrm{out}}_N-\om^2\mathbb{K}^{\mrm{out}}_N:\mX^{\mrm{out}}_N(\eps)\to(\mX^{\mrm{out}}_N(\eps))^{\ast}$ is Fredholm of index zero. Assuming that it is injective, the problem (\ref{MainPbE}) admits a unique solution $\u=c_{\u}\nabla s^++\tilde{\u}$. The goal of this paragraph is to derive a formula allowing one to compute $c_{\u}$ without knowing $\u$. Such kind of results are classical for scalar operators (see e.g. \cite{Gris92}, \cite[Theorem 6.4.4]{KoMR97},
 \cite{DNBL90a,DNBL90b,AsCS00,HaLo02,YoAS02,Nkem16}). They are used in particular for numerical purposes. But curiously they do not seem to exist for Maxwell's equations in 3D, not even for classical situations with positive materials in non smooth domains. We emphasize that the analysis we develop can be adapted to the latter case.\\
\newline
In order to establish the desired expression, for $\om\in\R$, first we introduce the field $\w_N\in\mX^{\mrm{out}}_N(\eps)$ such that 
\begin{equation}\label{defwSingu}
		\dsp\int_{\Om}\mu^{-1}\curl\v\cdot\curl\overline{\w_N}\,dx -\om^2 \fint_\Om \eps \v\cdot \overline{\w_N}\,dx
		= \int_\Om \eps\tilde{\v}\cdot\nabla \overline{s^+}\,dx,\qquad \forall\v\in\mX^{\mrm{out}}_N(\eps).
\end{equation}
Note that Problem (\ref{defwSingu}) is well-posed when $\mathbb{A}^{\mrm{out}}_N-\om^2\mathbb{K}^{\mrm{out}}_N$ is an isomorphism. Indeed, using (\ref{FormulaAdjoint}), one can check that it involves the operator $(\mathbb{A}^{\mrm{out}}_N-\om^2\mathbb{K}^{\mrm{out}}_N)^{\ast}$, that is the adjoint of $\mathbb{A}^{\mrm{out}}_N-\om^2\mathbb{K}^{\mrm{out}}_N$. Moreover $\v\mapsto \textstyle\int_\Om \eps\tilde{\v}\cdot\nabla\overline{s^+}\,dx$ is a linear form over $\mX^{\mrm{out}}_N(\eps)$.
\begin{theorem}\label{ThmCoefSinguE}
Assume that $\om\in\R$, Assumptions 1--3 are valid and $\mathbb{A}^{\mrm{out}}_N-\om^2\mathbb{K}^{\mrm{out}}_N:\mX^{\mrm{out}}_N(\eps)\to(\mX^{\mrm{out}}_N(\eps))^{\ast}$ is injective. Then the solution $\u=c_{\u}\nabla s^++\tilde{\u}$ of the electric problem (\ref{MainPbE}) is such that 
\begin{equation}\label{FormulaCoefSinguE}
c_{\u}=i\om\int_{\Om} \J\cdot\overline{\w_N}\,dx\bigg/\int_{\Om}\div(\eps\nabla s^{+})\,\overline{s^+}\,dx.
\end{equation}
Here $\w_N$ is the function which solves (\ref{defwSingu}).
\end{theorem}
\begin{remark}
Note that in practice $\w_N$ can be computed once for all because it does not depend on $\J$. Then the value of $c_{\u}$ can be determined very simply via Formula (\ref{FormulaCoefSinguE}).
\end{remark}
\begin{proof}
By definition of $\u$, we have
\[
\dsp\int_{\Om}\mu^{-1}\curl\u\cdot\curl\overline{\w_N}\,dx -\om^2 \fint_\Om \eps \u\cdot\overline{\w_N}\,dx= i\om\int_{\Om} \J\cdot\overline{\w_N}\,dx.
\]
On the other hand, from (\ref{defwSingu}), there holds
\[
\dsp\int_{\Om}\mu^{-1}\curl\u\cdot\curl\overline{\w_N}\,dx-\om^2 \fint_\Om \eps \u\cdot\overline{\w_N}\,dx= \int_\Om \eps\tilde{\u}\cdot\nabla\overline{s^+}\,dx.
\]
From these two relations as well as (\ref{defDivFaible-4}), we get
\[
i\om\int_{\Om} \J\cdot\overline{\w_N}\,dx=\int_\Om \eps\tilde{\u}\cdot\nabla\overline{s^+}\,dx=c_{\u}\int_{\Om}\div(\eps\nabla s^{+})\,\overline{s^+}\,dx.
\]
But Lemma \ref{lemmaNRJ} in Appendix guarantees that $\textstyle\Im m\,\int_{\Om}\div(\eps\nabla s^{+})\,\overline{s^+}\,dx\ne0$. Therefore we find the desired formula.
\end{proof}

\subsection{Limiting absorption principle}\label{ParagLimiting}
In \S\ref{subsec-mainanalysisE}, we have proved well-posedness of the problem for the electric field in the space $\mX^{\mrm{out}}_N(\eps)$. But up to now, we have not explained why we select this framework. In particular, as mentioned in \S\ref{subsec-hypersingularities}, well-posedness also holds in $\mX^{\mrm{in}}_N(\eps)$ where $\mX^{\mrm{in}}_N(\eps)$ is defined as $\mX^{\mrm{out}}_N(\eps)$ with $s^+$ replaced by $s^-$ (see (\ref{DefSinguEss}) for the definitions of $s^{\pm}$). In general, the solution in $\mX^{\mrm{in}}_N(\eps)$ differs from the one in $\mX^{\mrm{out}}_N(\eps)$. Therefore one can build infinitely many solutions of Maxwell's problem as linear interpolations of these two solutions. Then the question is: which solution is physically relevant? Classically, the answer can be obtained thanks to the limiting absorption principle. The idea is the following. In practice, the dielectric permittivity takes complex values, the imaginary part being related to the dissipative phenomena in the materials. Set 
\[
\eps^\delta:=\eps+i\delta
\]
where $\eps$ is defined as previously (see (\ref{sec-assumptions})) and  $\delta>0$ (the sign of $\delta$ depends on the convention for the time-harmonic dependence (in $e^{-i\om t}$ here)). Due to the imaginary part of $\eps^\delta$ which is uniformly positive, one recovers some coercivity properties which allow one to prove well-posedness of the corresponding problem for the electric field in the classical framework. The physically relevant solution for the problem with the real-valued $\eps$ then should be the limit of the sequence of solutions for the problems involving $\eps^\delta$  when  $\delta$ tends to zero. The goal of the present paragraph is to explain how to show that this limit is the solution of the problem set in $\mX^{\mrm{out}}_N(\eps)$. 

\subsubsection{Limiting absorption principle for the scalar case}

Our proof relies on a similar result for the 3D scalar problem which is the analogue of what has been done in 2D in \cite[Theorem 4.3]{BoCC14}. Consider  the problem
 \begin{equation}\label{pbclassic}
  \text{Find } \varphi^\delta\in\mH^1_0(\Om) \text{ such that } -\div(\eps^\delta \nabla \varphi^\delta)=f ,
 \end{equation}
 where $f\in(\mH^1_0(\Om))^\ast$. 
 Since $\delta>0$, by the Lax-Milgram lemma, this problem is   well-posed for all $f\in(\mH^1_0(\Om))^\ast$ and in particular for all $f\in(\mathring{\mV}^1_\beta(\Om))^\ast$, $\beta>0$.  
 Our objective is to prove that $ \varphi^\delta$ converges  when $\delta$ tends to zero to the unique solution of the problem
 \begin{equation} \label{eq-limitpbsca}
   \text{Find } \varphi\in \mathring{\mV}^{\mrm{out}}\text{ such that }  A^{\mrm{out}}_{\eps}\varphi=f.
 \end{equation}
We expect a convergence in a space $\mathring{\mV}^1_\beta(\Om)$ with $0<\beta<\beta_0$.
We first need a decomposition of $\varphi^\delta$ as a sum of a singular part and a regular part. Since problem \eqref{pbclassic} is strongly elliptic, one can directly apply the theory presented in \cite{KoMR97}. On the one hand, from the assumptions of Section \ref{sec-assumptions}, one can verify that for $\delta$ small enough, there exists one and only one singular exponent  $\lambda^\delta\in\C$  such that $\Re e\,\lambda^\delta\in(-1/2;-1/2+\beta_0-\sqrt{\delta})$. We denote by $\mathfrak{s}^\delta$ the corresponding singular function such that
\[
\mathfrak{s}^\delta(r,\theta,\varphi)=r^{\lambda^\delta}\, \Phi^\delta(\theta,\phi).
\]
Note that it satisfies $ \div(\eps^\delta \nabla \mathfrak{s}^\delta)=0$ in $\mathcal{K}$. As in \eqref{DefSinguEss} for $s^{\pm}$, we set 
\begin{equation}\label{singexprdelta}
s^\delta(x)=\chi(r)\, r^{-1/2+i\eta^\delta}\,\Phi^\delta(\theta,\phi),
\end{equation}
where $\eta^\delta\in\C$ is the number such that $\lambda^\delta=-1/2+i\eta^\delta$. By applying \cite[Theorem 5.4.1]{KoMR97}, we get the following result.
\begin{lemma}	\label{lem-decompphidelta}
Let $0<\beta<\beta_0$ and $f\in(\mathring{\mV}^1_\beta(\Om))^\ast$. The solution $\varphi^\delta$ of \eqref{pbclassic} decomposes as 
\begin{equation}\label{asymudel}
\varphi^\delta =c^\delta s^\delta+\tilde{\varphi}^\delta 
\end{equation}
where $c^\delta\in \C$ and $\tilde{\varphi}^\delta \in \mathring{\mV}^1_{-\beta}(\Om)$.
\end{lemma}
\noindent Let us first study the limit of the singular function.
\begin{lemma}
For all $\beta>0$, when $\delta$  tends to zero, the function $s^\delta$ converges in $\mathring{\mV}^1_\beta(\Om)$ to $s^+$ and not to $s^-$ (see the definitions in \eqref{DefSinguEss}).
 \end{lemma}
\begin{proof}
The pair $(\Phi^\delta,\lambda^\delta)$ solves the spectral problem
\begin{equation}\label{PbSpectralDelta}
\begin{array}{|l}
\mbox{Find }(\Phi^\delta,\lambda^\delta)\in\mH^1(\mathbb{S}^2)\setminus\{0\}\times\C\mbox{ such that }\\
\dsp\int_{\mathbb{S}^2}\eps^\delta\nabla_S\Phi^\delta\cdot\nabla_S\overline{\Psi}\,ds 
= \lambda^\delta(\lambda^\delta+1)\int_{\mathbb{S}^2}\eps^\delta\Phi^\delta\,\overline{\Psi}\,ds,\qquad \forall\Psi\in\mH^1(\mathbb{S}^2).
\end{array}
\end{equation}
Postulating the expansions $\Phi^\delta=\Phi^0 +\delta\Phi'+\dots$, $\lambda^\delta=\lambda^0+\delta\lambda'+\dots$ in this problem and identifying the terms in $\delta^0$, we get $\Phi^0=\Phi$ and we find that $\lambda^0=-1/2+i\eta^0$ where $\eta^0$ coincides with $\eta$ or $-\eta$ (see an illustration with Figure \ref{FigureLimitingAbsorptionPrinciple}). At order $\delta$, we get the variational equality 
\begin{equation}\label{PbLim2}
\begin{array}{rcl}
\dsp\int_{\mathbb{S}^2}\eps\nabla_S\Phi'\cdot\nabla_S\overline{\Psi}\,ds +i\dsp\int_{\mathbb{S}^2}\nabla_S\Phi\cdot\nabla_S\overline{\Psi}\,ds 
&\hspace{-0.25cm}=&\hspace{-0.25cm} \dsp\lambda^0(\lambda^0+1)\bigg(\int_{\mathbb{S}^2}\eps\Phi'\,\overline{\Psi}\,ds+i\int_{\mathbb{S}^2}\Phi\,\overline{\Psi}\,ds\bigg)\hspace{-0.5cm}\\[10pt]
 & & \hspace{-1cm}\dsp+\lambda'(2\lambda^0+1)\int_{\mathbb{S}^2}\eps\Phi\,\overline{\Psi}\,ds,\qquad \forall\Psi\in\mH^1(\mathbb{S}^2).
\end{array}
\end{equation}
Taking $\Psi=\Phi$ in (\ref{PbLim2}), using (\ref{spectralPb}) and observing that $\lambda^0(\lambda^0+1)=-\eta^2-1/4$, this implies 
\[
\dsp\int_{\mathbb{S}^2}|\nabla_S\Phi|^2+(\eta^2+1/4)|\Phi|^2\,ds=\lambda'2\eta^0\int_{\mathbb{S}^2}\eps|\Phi|^2\,ds.
\]
Thus $\lambda'$ is real. Since by definition of $\lambda^{\delta}$, we have $\Re e\,\lambda^{\delta}>-1/2$ for $\delta>0$, we infer that $\lambda'>0$. As a consequence, we have
\[
\eta^0\int_{\mathbb{S}^2}\eps|\Phi|^2\,ds>0
\]
which according to the definition of $\eta$ in (\ref{eq-intPhineq0}) ensures that $\eta^0=\eta$. Therefore the pointwise limit of $s^\delta$ when $\delta$ tends to zero is indeed $s^+$ and not $s^-$. This is enough to conclude that $s^\delta$ converges to $s^+$ in  $\mathring{\mV}^1_\beta(\Om)$ for $\beta>0$.
\end{proof}

\begin{figure}[!ht]
\centering
\begin{tikzpicture}[scale=1.3]
\draw[draw=black,line width=1pt,->](-3,0)--(2,0);
\draw[dashed,line width=1pt](-1.5,-1.5)--(-1.5,1.5);
\draw[draw=black,line width=1pt,->](0,-1.5)--(0,2);
\filldraw [red,draw=none] (-1.5,-0.965) circle (0.1);
\filldraw [red,draw=none] (-1.5,0.965) circle (0.1);
\draw[xscale=3] (-0.313,-0.965) node[cross] {};
\draw[xscale=3] (-0.374,-0.958) node[cross] {};
\draw[xscale=3] (-0.436,-0.963) node[cross] {};
\draw[xscale=3] (-0.487,-0.965) node[cross] {};

\draw[xscale=3] (-0.687,0.965) node[cross] {};
\draw[xscale=3] (-0.626,0.958) node[cross] {};
\draw[xscale=3] (-0.564,0.958) node[cross] {};
\draw[xscale=3] (-0.514,0.963) node[cross] {};

\node at (-1.3,1) [red,anchor=west] {$-\lambda^0$};
\node at (-1.7,-1) [red,anchor=east] {$\lambda^0$};
\node at (2.4,0) [anchor=west] {$\Re e\,\lambda$};
\node at (0.7,1.7) [anchor=south] {$\Im m\,\lambda$};
\node at (-1.6,-0.3) [anchor=center] {\small $ -1/2$};
\node at (-2,1.5) [anchor=south] {\small \textcolor{blue}{$-\lambda^\delta-1$ when $\delta\to0^+$}};
\node at (-1,-1.5) [anchor=north] {\small \textcolor{blue}{$\lambda^\delta$ when $\delta\to0^+$}};
\begin{scope}[xshift=-1.2cm,yshift=-0.2cm]
\draw[blue,dotted,->] (-1.2,1.4) .. controls (-0.85,1.5) .. (-0.35,1.49);
\end{scope}
\begin{scope}[xshift=-1.8cm,yshift=0.2cm]
\draw[blue,dotted,->] (1.2,-1.4) .. controls (0.85,-1.5) .. (0.35,-1.49);
\end{scope}
\end{tikzpicture}\qquad
\raisebox{2.5cm}{$\renewcommand{\arraystretch}{1.4} \begin{array}{|c|c|}
\hline
\delta & \lambda^{\delta}\\\hline
0 &  -0.5-0.965i\\\hline
0.001 &  -0.498-0.965i\\\hline
0.01 &  -0.487-0.965i\\\hline
0.05 &  -0.436-0.963i\\\hline
0.1 &  -0.374-0.958i\\\hline
\end{array}$}
\caption{Behaviour of the eigenvalue $\lambda^\delta$ close to the line $\Re e\,\lambda=-1/2$ as the dissipation $\delta$ tends to zero. Here the values have been obtained solving the problem (\ref{PbSpectralDelta}) with a Finite Element Method. We work in the conical tip defined via (\ref{ConicalTip}) with $\alpha=2\pi/3$ and $\kappa_\eps=-1.9$.\label{FigureLimitingAbsorptionPrinciple}}
\end{figure}
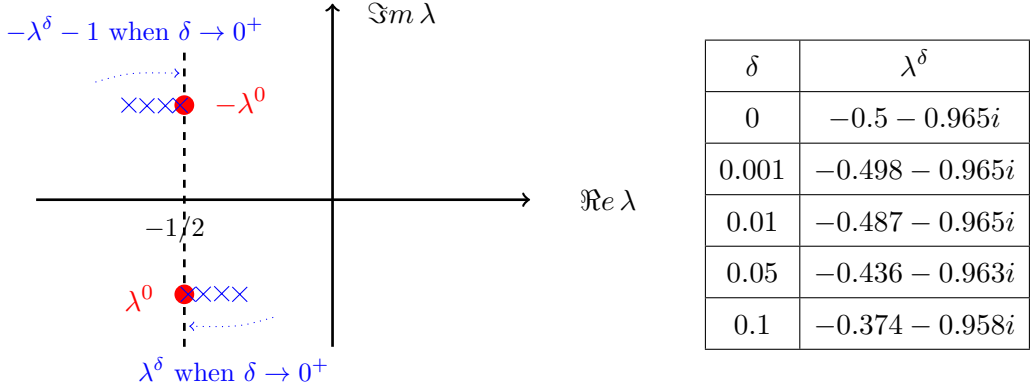

\noindent Then proceeding exactly as in the proof of \cite[Theorem 4.3]{BonCheCla13}, one can establish the following result.
\begin{lemma}\label{Abslimscalaire}
Let $0<\beta<\beta_0$ and $f\in(\mathring{\mV}^1_\beta(\Om))^\ast$.  If Assumption 3  holds, then $(\varphi^\delta=c^{\delta}\,s^\delta+\tilde{\varphi}^\delta)$ converges to $\varphi=c\,s^++\tilde{\varphi}$ in $\mathring{\mV}^1_\beta(\Om)$ as $\delta$ tends to zero. Moreover, $(c^\delta,\tilde{\varphi}^\delta)$ converges to $(c,\tilde{\varphi})$ in $\C\times\mathring{\mV}^1_{-\beta}(\Om)$. In this statement, $\varphi^\delta$ (resp. $\varphi$) is the solution of \eqref{pbclassic} (resp. \eqref{eq-limitpbsca}). 
 \end{lemma}
\noindent Note that the results of Lemma \ref{Abslimscalaire} still hold if we  replace $f$ by a family of source terms $(f^\delta)\in(\mathring{\mV}^1_\beta(\Om))^\ast$ that converges to $f$ in $(\mathring{\mV}^1_\beta(\Om))^\ast$ when $\delta$ tends to zero.

\subsubsection{Limiting absorption principle for the electric problem} 
 The problem
\begin{equation}\label{pb-udelta}
 \text{Find}~\u^\delta \in  \mX_N(\eps^\delta)~\text{such that}~~ \curl\mu^{-1}\curl \u^\delta-\om^2\eps^\delta\u^\delta=i\om \J , 
\end{equation}
with $\mX_N(\eps^\delta)= \{\boldsymbol{E}\in \Hspace_N(\curl)\,|\,\div(\eps^\delta\boldsymbol{E})=0\}$, is well-posed for all $\om\in\R$ and all $\delta>0$. This result is classical when $\mu$ takes positive values while it can be shown by using \cite{BoCC14} when $\mu$ changes sign. We want to study the convergence of $\u^\delta$ when $\delta$ goes to zero. Let $(\delta_n)$ be a sequence of positive numbers such that $\textstyle\lim_{n\to+\infty}\delta_n=0$. To simplify, we denote the quantities with an index $n$ instead of $\delta_n$ (for example we write $\eps^n$ instead of $\eps^{\delta_n}$). 
 \begin{lemma}\label{lem-udeltan}
Suppose that $(\u^{n})$ is a sequence of elements of $\mX_N(\eps^n)$ such that $(\curl \u^{n})$ is bounded in $\Lspace^2(\Om)$. Then, under Assumption 3, for all $\beta$ satisfying \eqref{betainfbeta_0et1demi}, for all $n\in\N$, $\u^n$ admits the decomposition $\u^n=c^n\nabla s^n+\tilde{\u}^n$  with $c^n\in\mathbb{C}$ and $\tilde{\u}^n\in \boldsymbol{\mV}^0_{-\beta}(\Om)$. Moreover, there exists a subsequence such that $(c^n)$ converges to some $c$ in $\C$ while $(\tilde{\u}^n)$ converges to some $\tilde{\u}$ in $\boldsymbol{\mV}^0_{-\beta}(\Om)$. Finally, the field $\u:=c\nabla s^++\tilde{\u}$ belongs to $\mX^{\mrm{out}}_N(\eps)$.
 \end{lemma}
\begin{proof}
For all $n\in\N$, we have $\u^n\in\mX_N(\eps^\delta)\subset\Lspace^2(\Om)$. Therefore, there exist $\varphi^n\in\mH^1_0(\Om)$ and $\boldsymbol{\psi}^n\in\mX_T(1)$, satisfying $\curl\boldsymbol{\psi^n}\times \nu=0$ on $\partial\Omega$ such that	$\u^n=\nabla\varphi^n+\curl \boldsymbol{\psi}^n$. Moreover, we have the estimate 
\[
\|\Delta \boldsymbol{\psi^n}\|_\Omega=\|\curl \u^{n}\|_\Omega\leq C.
\]
As a consequence, Proposition \ref{propoLaplacienVect} guarantees that  $(\curl\boldsymbol{\psi^n})$ is a bounded sequence of $\boldsymbol{\mV}^0_{-\beta}(\Om)$, and Proposition \ref{propoLaplacienVectcompact} ensures that there exists a subsequence such that  $(\curl\boldsymbol{\psi^n})$ converges in $\boldsymbol{\mV}^0_{-\beta}(\Om)$. Now from the fact that $\div(\eps^{n}\u^n)=0$, we obtain 
\[
\div(\eps^{n} \nabla\varphi^n)=-\div(\eps^{n} \curl\boldsymbol{\psi}^n)\in(\mathring{\mV}^1_\beta(\Om))^\ast.
\]
By Lemmas \ref{lem-decompphidelta} and \ref{Abslimscalaire}, this implies that the function  $\varphi^n$ decomposes as 
	 $\varphi^n=c^ns^{n} +\tilde{\varphi}^n$ with 
	 $c^n\in\mathbb{C}$ and $\tilde{\varphi}^n\in\mathring{\mV}^1_{-\beta}(\Om)$. Moreover, $(c^n)$ converges to $c$ in $\C$ while $(\tilde{\varphi}^n)$ converges to $\tilde{\varphi}$ in $\mV^1_{-\beta}(\Om)$.\\
\newline
Summing up, we have that $\u^n=c^n\nabla s^{n}+\tilde{\u}^n$ where  $\tilde{\u}^n=\nabla\tilde{\varphi}^n+\curl \boldsymbol{\psi}^n$ converges to $\tilde{\u}$ in $\boldsymbol{\mV}^0_{-\beta}(\Om)$. In particular, this implies that $\u^n$ converges to $\u= c\nabla s^++\tilde{\u}$ in $\boldsymbol{\mV}^0_{\gamma}(\Om)$ for all $\gamma>0$. It remains to prove that $\u \in \mX^{\mrm{out}}_N(\eps) $, which amounts to show that $\u$ satisfies \eqref{defDivFaible-4}. To proceed, we take the limit as $n\rightarrow +\infty$ in the identity
\[
-c^n\int_{\Om}\div(\eps^{n}\nabla s^{n})\varphi \,dx+\int_{\Om}\eps^{n}\tilde{\u}^n\cdot\nabla\varphi\,dx=0
\]
which holds for all $\varphi\in \mathring{\mV}^1_\beta(\Om)$ because $\u^n\in\mX_N(\eps^{n})$. 
\end{proof}
\begin{theorem} Let $\om\in\R$.  
	Suppose that   Assumptions 1, 2 and 3 hold, and that $\u=0$ is the only function of $\mX_N(\eps)$ satisfying 
	\begin{equation}
	\label{PbhomogdansXN}
\curl\mu^{-1}\curl \u-\om^2\eps\u=0.
	\end{equation}
Then the sequence of solutions $(\u^\delta=c^\delta\nabla s^{\delta}+\tilde{\u}^\delta)$ of \eqref{pb-udelta} converges, as $\delta$ tends to 0, to the unique solution $\u=c \nabla s^++\tilde{\u}\in\mX^{\mrm{out}}_N(\eps) $ of (\ref{MainPbE}) in the following sense: $(c_\delta)$ converges to $c$ in $\C$, $(\tilde{\u}^\delta)$ converges to $\tilde{\u}$ in $\boldsymbol{\mV}^0_{-\beta}(\Om)$ and 
$(\curl \u^\delta)$ converges to $\curl \u$ in $\Lspace^2(\Om)$.
\end{theorem}
\begin{proof}
Let $(\delta_n)$ be a sequence of positive numbers such that $\textstyle\lim_{n\to+\infty}\delta_n=0$. Denote by $\u^n$ the unique function of $\mX_N(\eps^{n})$ such that 
\begin{equation}
\label{eq-suiteu_n}
\curl\mu^{-1}\curl \u^n-\om^2\eps^{n}\u^n=i\om \J. 
\end{equation}
Note that we set again $\eps^n$ instead of $\eps^{\delta_n}$. The proof is in two steps. First, we establish the desired property by assuming that $	(\|\curl \u^n\|_\Omega)$ is bounded. Then we show that this hypothesis is indeed satisfied.
	\\	
	{\bf 	First step.} Assume that there is a constant $C>0$ such that for all $n\in\N$
		\begin{equation}\label{eq-curlun}
		\|\curl \u^n\|_\Omega\leq C.
		\end{equation}
		By lemma \ref{lem-udeltan}, we can extract a subsequence from $(\u^n=c^n\nabla s^{n}+\tilde{\u}^n)$ such that $(c^n)$ converges to $c$ in $\C$, $(\tilde{\u}^n)$ 	converges to $\tilde{\u}$ in $\boldsymbol{\mV}^0_{-\beta}(\Om)$, with $\u=\tilde{u}+c\nabla s^+\in \mX^{\mrm{out}}_N(\eps)$. Besides, since for all $n\in\N$, $\curl\u^n\in\Lspace^2(\Om)$, there exist $h^n\in\mH^1_\#(\Om)$ and $\boldsymbol{\w}^n\in\mX_N(1)$, such that	
\begin{equation}\label{helm-decomp-mu}
\mu^{-1}\curl\u^n=\nabla h^n+\curl \boldsymbol{\w}^n.
\end{equation}
Observing that $(\w^n)$ is bounded in $\mX_N(1)$, from Lemma \ref{LemmaWeightedClaBis}, we deduce that it admits a subsequence which converges in $\boldsymbol{\mV}^0_{-\beta}(\Om)$. Multiplying \eqref{eq-suiteu_n} taken for two indices $n$ and $m$  by $\overline{(\w^n-\w^m)}$, and integrating by parts, we obtain 
\[
\int_\Om|\curl\w^n-\curl\w^m|^2\,dx=\om^2\int_\Om\left(\eps^{n}\u^n-\eps^{m}\u^m\right)\overline{(\w^n-\w^m)}\,dx. 
\]
This implies that $(\curl\w^n)$ converges in $\Lspace^2(\Om)$. Then, from \eqref{helm-decomp-mu}, we deduce that
\[
\div  \left(\mu \nabla h^n\right)=-\div \left(\mu\, \curl\w^n\right)\mbox{ in }\Om.
\]
By Assumption 2, the operator $A_{\mu}: \mH_\#^1(\Om)\to (\mH_\#^1(\Om))^\ast$ is an isomorphism. Therefore $(\nabla h^n)$ converges in $\Lspace^2(\Om)$. From \eqref{helm-decomp-mu}, this shows that  $(\curl\u^n)$ converges to $\curl \u$ in $\Lspace^2(\Om)$. Finally, we know that $\u^n$ satisfies
\[
\int_{\Om}\mu^{-1}\curl\u^n\cdot\curl\overline{{\v}}\,dx -\om^2 \int_\Om \eps^{n} \u^n\cdot\overline{{\v}}\,dx= i\om\int_\Om \J\cdot \overline{{\v}}\,dx
\]
for all ${\v}\in\boldsymbol{\mV}^0_{-\beta}(\Om)$. 
				Taking the limit, we get that  $\u$ satisfies
				\begin{equation}
				\label{eq-variationnel-u}
				\int_{\Om}\mu^{-1}\curl\u\cdot\curl\overline{{\v}}\,dx -\om^2 \fint_\Om \eps \u\cdot\overline{{\v}}\,dx= i\om\int_\Om \J\cdot \overline{{\v}}\,dx
				\end{equation}
					for all ${\v}\in\boldsymbol{\mV}^0_{-\beta}(\Om)$. 
			Since in addition, $\u$ satisfies \eqref{defDivFaible-4}, \eqref{eq-variationnel-u} also holds for $\v=\nabla s^+$ and we get that $\u$ is  the unique solution   $u$ of (\ref{MainPbE}). 
\\		
\textbf{Second step.} Now we prove that the assumption \eqref{eq-curlun} is satisfied. Suppose by contradiction that there exists a subsequence of $(\u^n)$ such that 	
\[
\|\curl \u^n\|_\Omega\rightarrow +\infty
\]
and consider the sequence $(\v^n)$ with for all $n\in\N$, 	$\v^n:=\u^n/\|\curl \u^n\|_\Omega$. We have
	\begin{equation}
	\label{eq-suitev_n}
	\v^n \in  \mX_N(\eps^{n})\quad\mbox{ and }\quad\curl\mu^{-1}\curl \v^n-\om^2\eps^{n}\v^n=i\om \J/\|\curl \u^n\|_\Omega. 
	\end{equation}
Following the first step of the proof, we find that we can extract a subsequence from $(\v^n)$ which converges, in the sense given in the theorem, to the unique solution of  the homogeneous problem (\ref{MainPbE}) with $\J=0$. But by Proposition \ref{trappedmodes}, this solution also solves \eqref{PbhomogdansXN}. As a consequence, it is equal to zero. In particular, it implies that $(\curl \v^n)$ converges to zero in $\Lspace^2(\Om)$, which is impossible since by construction, for all $n\in\N$, we have $\|\curl \v^n\|_\Omega=1$.
\end{proof}
\section{Analysis of the problem for the magnetic component}\label{SectionChampH}

In this section, we turn our attention to the analysis of the Maxwell's problem for the magnetic component. Importantly, in the whole section, we suppose that  $\beta$ satisfies (\ref{betainfbeta_0et1demi}), that is $0<\beta<\min(1/2,\beta_0)$. Contrary to the analysis for the electric component, we define functional spaces which depend on $\beta$:
\[
\mZ^{\mrm{out}}_T(\mu):=\{\u\in\Lspace^2(\Omega)\,|\,\curl\u\in\mrm{span}(\eps\nabla s^{+})\oplus\boldsymbol{\mV}^0_{-\beta}(\Om)
,\,\div(\mu\u)=0\mbox{ in }\Om,\,\mu\u\cdot\nu=0\mbox{ on }\partial\Om\}
\]
and for $\xi\in\mL^\infty(\Om)$, 
\[
\mZ^{\pm\beta}_T(\xi):=\{ \u\in\Lspace^2(\Om)\,|\,\curl\boldsymbol{\u}\in \boldsymbol{\mV}^0_{\pm\beta}(\Om),\,\div\,(\xi \u)=0\mbox{ in }\Om\mbox{ and }\xi\u\cdot\nu=0 \mbox{ on }\partial\Om \}.
\]
Note that we have $\mZ^{-\beta}_T(\mu)\subset\mZ^{\mrm{out}}_T(\mu)\subset \mZ^{\beta}_T(\mu)$. The conditions $\div(\mu\u)=0$ in $\Om$ and $\mu\u\cdot\nu=0$ on $\partial\Om$ for the elements of these spaces boil down to impose 
\[
\int_{\Om}\mu\u\cdot\nabla\varphi\,dx=0,\qquad\forall \varphi\in\mH^1_{\#}(\Om).
\]
\begin{remark}
Observe that the elements of $\mZ^{\mrm{out}}_T(\mu)$ are in $\Lspace^2(\Omega)$ but have a singular curl. On the other hand, the elements of $\mX^{\mrm{out}}_N(\eps)$ are singular but have a curl in $\Lspace^2(\Omega)$. This is coherent with the fact that for the situations we are considering in this work, the electric field is singular while the magnetic field is not.
\end{remark}
\noindent The analysis of the problem for the magnetic component leads to consider the formulation 
\begin{equation}\label{MainPbH}
\begin{array}{|l}
\mbox{Find }\u\in\mZ^{\mrm{out}}_T(\mu)\mbox{ such that }\\
\dsp\fintH_{\Om}\eps^{-1}\curl\u\cdot\curl\overline{\v}\,dx -\om^2 \int_\Om \mu \u\cdot\overline{\v}= \int_\Om \eps^{-1}\J\cdot\curl\overline{\v},\qquad \forall\v\in\mZ^{\beta}_T(\mu),

\end{array}
\end{equation}
where $\J\in\boldsymbol{\mV}^0_{-\beta}(\Om)$. 
Again, the first integral in the left-hand side of (\ref{MainPbH}) is not a classical integral. Similarly to definition (\ref{defDivFaible-4}), we set
\[
\fintH_{\Om}\nabla s^+\cdot\curl\overline{\v}\,dx:=0,\qquad\forall \v\in\mZ^{\beta}_T(\mu).
\]
As a consequence, for $\u\in \mZ^{\mrm{out}}_T(\mu)$ such that $\curl\u=c_{\u}\,\eps\nabla s^++\boldsymbol{\zeta_{\u}}$ (we shall use this notation throughout the section) and $\v\in\mZ^{\beta}_T(\mu)$, there holds
\begin{equation}\label{eq-intH-symZero}
\fintH_{\Om}\eps^{-1}\curl\u\cdot\curl\overline{\v}\,dx=\int_{\Om}\eps^{-1}\boldsymbol{\zeta_{\u}}\cdot\curl\overline{\v}\,dx.
\end{equation}
Note that for $\u$, $\v$ in $\mZ^{\mrm{out}}_T(\mu)$ such that $\curl\u=c_{\u}\,\eps\nabla s^++\boldsymbol{\zeta_{\u}}$, $\curl\v=c_{\v}\,\eps\nabla s^++\boldsymbol{\zeta_{\v}}$, we have 
\begin{equation}\label{eq-intH-sym}
\begin{array}{rcl}
\fintH_{\Om}\eps^{-1}\curl\u\cdot\curl\overline{\v}\,dx&=&\int_{\Om}\eps^{-1}\boldsymbol{\zeta_{\u}}\cdot(\overline{c_{\v}\,\eps\nabla s^++\boldsymbol{\zeta_{\v}}})\,dx \\[10pt]
&=&\int_{\Om}\eps^{-1}\boldsymbol{\zeta_{\u}}\cdot\overline{\boldsymbol{\zeta_{\v}}}\,dx-\overline{c_{\v}}\int_{\Om}\div(\boldsymbol{\zeta_{\u}})\,\overline{s^{+}}\,dx \\[10pt]
&=&\int_{\Om}\eps^{-1}\boldsymbol{\zeta_{\u}}\cdot\overline{\boldsymbol{\zeta_{\v}}}\,dx+c_{\u}\overline{c_{\v}}\int_{\Om}\div(\eps\nabla s^+)\,\overline{s^{+}}\,dx.
\end{array}
\end{equation}
We denote by $a_T(\cdot,\cdot)$ (resp. $\ell_T(\cdot)$) the sesquilinear form (resp. the antilinear form) appearing in the left-hand side (resp. right-hand side) of (\ref{MainPbH}).

\begin{remark}
Note that in (\ref{MainPbH}), the solution and the test functions do not belong to the same space. This is different from the formulation (\ref{MainPbE}) for the electric field but seems necessary in the analysis below to obtain a well-posed problem (in particular to prove Proposition \ref{propoUniquenessH}). Note also that even if the functional framework depends on $\beta$, the solution will not if $\J$ is regular enough (see the explanations in Remark \ref{frameworkIndbeta}). 
\end{remark}

\subsection{Equivalent formulation}

Define the spaces
\[
\begin{array}{lcl}
\Hspace^{\beta}(\curl)&\hspace{-0.2cm}:=&\hspace{-0.2cm}\{\u\in\Lspace^2(\Omega)\,|\,\curl\u\in\boldsymbol{\mV}^0_{\beta}(\Om)\}\\
\Hspace^{\mrm{out}}(\curl)&\hspace{-0.2cm}:=&\hspace{-0.2cm}\{\u\in\Lspace^2(\Omega)\,|\,\curl\u\in\mrm{span}(\eps\nabla s^{+})\oplus\boldsymbol{\mV}^0_{-\beta}(\Om)\}.
\end{array}
\]

\begin{lemma}\label{EquivHcurlH}
Under Assumptions 1--2, the field $\u$ is a solution of (\ref{MainPbH}) if and only if it solves the problem 
\begin{equation}\label{Formu2H}
\begin{array}{|l}
\mbox{Find }\u\in \Hspace^{\mrm{out}}(\curl)\mbox{ such that }\\[2pt]
a_T(\u,\v)= \ell_T(\v),\quad\forall \v \in \Hspace^{\beta}(\curl).
\end{array}
\end{equation}
\end{lemma}
\begin{proof}
If $\u$ satisfies (\ref{Formu2H}), then taking $\v=\nabla\varphi$ with $\varphi\in\mH^1_{\#}(\Om)$ in (\ref{Formu2H}), we get that $\u\in\mZ^{\mrm{out}}_T(\mu)$. This proves that $\u$ solves (\ref{MainPbH}).\\
\newline
Assume now that $\u$ is a solution of (\ref{MainPbH}). Let $\v$ be an element of $\Hspace^{\beta}(\curl)$. Introduce $\varphi\in\mH^1_{\#}(\Om)$ the function such that
\[
\int_{\Om}\mu\nabla\varphi\cdot\nabla\varphi'\,dx=\int_{\Om}\mu\v \cdot\nabla\varphi'\,dx,\qquad\forall \varphi'\in\mH^1_{\#}(\Om).
\]
The field $\hat{\v}:=\v-\nabla\varphi$ belongs to $\mZ^{\beta}_T(\mu)$.
Moreover, there holds $\curl\hat{\v}=\curl\v$ and since for $\u\in\mZ^{\mrm{out}}_T(\mu)$, we have
\[
\int_{\Om}\mu\u\cdot\nabla\varphi\,dx=0,\qquad\forall\varphi\in\mH^1_{\#}(\Om),
\]
we deduce that $a_T(\u,\v)=a_T(\u,\hat{\v})=\ell_T(\hat{\v})=\ell_T(\v)$. 
\end{proof}

\subsection{Norms in $\mZ^{\pm\beta}_T(\mu)$ and $\mZ^{\mrm{out}}_T(\mu)$}

We endow the space $\mZ^{\beta}_T(\mu)$ with the norm
\[
\|\u\|_{\mZ^{\beta}_T(\mu)}=(\|\u\|^2_{\Om}+\|\curl\u\|^2_{\boldsymbol{\mV}^0_{\beta}(\Om)})^{1/2},
\]
so that it is a Banach space. 
\begin{lemma}\label{LemmaNormeEquivHPlus}
Under Assumptions 1--2, there is a constant $C>0$ such that for all $\u\in\mZ^{\beta}_T(\mu)$, we have
\[
\|\u\|_{\Om}\le C\,\|\curl\u\|_{\boldsymbol{\mV}^0_{\beta}(\Om)}.
\]
As a consequence, the norm $\|\cdot\|_{\mZ^{\beta}_T(\mu)}$ is equivalent to the norm $\|\curl\cdot\|_{\boldsymbol{\mV}^0_{\beta}(\Om)}$ in $\mZ^{\beta}_T(\mu)$.
\end{lemma}
\begin{remark}
The result of Lemma \ref{LemmaNormeEquivHPlus} holds for all $\beta$ such that $0\le \beta < 1/2$  and not only for $0<\beta<\min(1/2,\beta_0)$.
\end{remark}
\begin{proof}
Let $\u$ be an element of $\mZ^{\beta}_T(\mu)$. Since $\u$ belongs to $\Lspace^2(\Om)$, according to the item $v)$ of Proposition \ref{propoPotential}, there are $\varphi\in\mH^1_{\#}(\Om)$ and $\boldsymbol{\psi}\in \mX_N(1)$ such that 
\begin{equation}\label{DecompoInter}
\u=\nabla\varphi+\curl\boldsymbol{\psi}.
\end{equation}
Lemma \ref{LemmaWeightedClaBis} guarantees that $\boldsymbol{\psi}\in\boldsymbol{\mV}^0_{-\beta}(\Om)$ with the estimate 
\begin{equation}\label{EstimaWeightedClaInterm}
\|\boldsymbol{\psi}\|_{\boldsymbol{\mV}^0_{-\beta}(\Om)} \le C\,\|\curl\boldsymbol{\psi}\|_{\Om}.
\end{equation}
Multiplying the equation $\curl\curl\boldsymbol{\psi}=\curl\u$ in $\Om$ by $\boldsymbol{\psi}$ and integrating by parts, we get 
\begin{equation}\label{EstimaWeightedClaInterm2}
\|\curl\boldsymbol{\psi}\|^2_{\Om}\le \|\curl\u\|_{\boldsymbol{\mV}^0_{\beta}(\Om)}\|\boldsymbol{\psi}\|_{\boldsymbol{\mV}^0_{-\beta}(\Om)}.
\end{equation}
Gathering (\ref{EstimaWeightedClaInterm}) and (\ref{EstimaWeightedClaInterm2}) leads to
\begin{equation}\label{EstimPlusInter}
\|\curl\boldsymbol{\psi}\|_{\Om}\le C\,\|\curl\u\|_{\boldsymbol{\mV}^0_{\beta}(\Om)}.
\end{equation}
On the other hand, using that 
\[
\int_{\Om}\mu\u\cdot\nabla\varphi'\,dx=0,\qquad \forall\varphi'\in\mH^1_{\#}(\Om)
\] 
and that $A_{\mu}:\mH^1_{\#}(\Om)\to(\mH^1_{\#}(\Om))^{\ast}$ is an isomorphism, we deduce that $\|\nabla\varphi\|_{\Om} \le C\,\|\curl\boldsymbol{\psi}\|_{\Om}$. Using this estimate and (\ref{EstimPlusInter}) in the decomposition (\ref{DecompoInter}), finally we obtain the desired result.
\end{proof}
\noindent If $\u\in\mZ^{\mrm{out}}_T(\mu)$, we have $\curl\u=c_{\u}\,\eps\nabla s^++\boldsymbol{\zeta_{\u}}$ with $c_{\u}\in\Cplx$ and $\boldsymbol{\zeta_{\u}}\in\boldsymbol{\mV}^0_{-\beta}(\Om)$. We endow the space $\mZ^{\mrm{out}}_T(\mu)$ with the norm
\[
\|\u\|_{\mZ^{\mrm{out}}_T(\mu)}=(\|\u\|^2_{\Om}+|c_{\u}|^2+\|\boldsymbol{\zeta_{\u}}\|^2_{\boldsymbol{\mV}^0_{-\beta}(\Om)})^{1/2},
\]
so that it is a Banach space.
\begin{lemma}\label{LemmaNormeEquivHOut}
Under Assumptions 1--3, there is $C>0$ such that for all $\u\in\mZ^{\mrm{out}}_T(\mu)$, we have
\begin{equation}\label{MainEstimaSpaceOut}
\|\u\|_{\Om}+|c_{\u}|\le C\,\|\boldsymbol{\zeta_{\u}}\|_{\boldsymbol{\mV}^0_{-\beta}(\Om)}.
\end{equation}
As a consequence, the norm $\|\u\|_{\mZ^{\mrm{out}}_T(\mu)}$ is equivalent to the norm $\|\boldsymbol{\zeta_{\u}}\|_{\boldsymbol{\mV}^0_{-\beta}(\Om)}$ for $\u\in\mZ^{\mrm{out}}_T(\mu)$.
\end{lemma}
\begin{proof}
Let $\u$ be an element of $\mZ^{\mrm{out}}_T(\mu)$. Since $\mZ^{\mrm{out}}_T(\mu)\subset\mZ^{\beta}_T(\mu)$, Lemma \ref{LemmaNormeEquivHPlus} provides the estimate
\begin{equation}\label{MainEstimaSpaceOutInterm}
\|\u\|_{\Om}\le C\,\|\curl\u\|_{\boldsymbol{\mV}^0_{\beta}(\Om)}\le C\,(|c_{\u}|+\|\boldsymbol{\zeta_{\u}}\|_{\boldsymbol{\mV}^0_{-\beta}(\Om)}).
\end{equation}
On the other hand, taking the divergence of $\curl\u=c_{\u}\,\eps\nabla s^++\boldsymbol{\zeta_{\u}}$, we obtain $c_{\u}\,\div(\eps\nabla s^+)=-\div\,\boldsymbol{\zeta_{\u}}$. Using the fact that $A^{\mrm{out}}_{\eps}:\mathring{\mV}^{\mrm{out}}\to(\mathring{\mV}^1_{\beta}(\Om))^{\ast}$ is an isomorphism, we get 
\[
|c_{\u}|\le C\,\|\div\,\boldsymbol{\zeta_{\u}}\|_{(\mathring{\mV}^1_{\beta}(\Om))^{\ast}}\le C\,\|\boldsymbol{\zeta_{\u}}\|_{\boldsymbol{\mV}^0_{-\beta}(\Om)}.
\]
Using this inequality in (\ref{MainEstimaSpaceOutInterm}) leads to (\ref{MainEstimaSpaceOut}).
\end{proof}

\subsection{Main analysis for the magnetic field}
\noindent Define the continuous operators $\mathbb{A}^{\mrm{out}}_T:\mZ^{\mrm{out}}_T(\mu)\to(\mZ^{\beta}_T(\mu))^{\ast}$ and $\mathbb{K}^{\mrm{out}}_T:\mZ^{\mrm{out}}_T(\mu)\to(\mZ^{\beta}_T(\mu))^{\ast}$ such that for all $\u\in\mZ^{\mrm{out}}_T(\mu)$, $\v\in\mZ^{\beta}_T(\mu)$, 
\begin{equation}\label{DefOpChampH}
\langle\mathbb{A}^{\mrm{out}}_T\u,\v\rangle = \fintH_{\Om}\eps^{-1}\curl\u\cdot\curl\overline{\v}\,dx,\qquad\qquad\langle\mathbb{K}^{\mrm{out}}_T\u,\v\rangle = \int_{\Om}\mu\u\cdot\overline{\v}\,dx.
\end{equation}
With this notation, we have $\langle(\mathbb{A}^{\mrm{out}}_T-\om^2\mathbb{K}^{\mrm{out}}_T)\u,\v\rangle=a_{T}(\u,\v)$.

\begin{proposition}\label{propoUniquenessH}
Under Assumptions 1--3, the operator $\mathbb{A}^{\mrm{out}}_T:\mZ^{\mrm{out}}_T(\mu)\to(\mZ^{\beta}_T(\mu))^{\ast}$ is an isomorphism. 
\end{proposition}
\begin{proof}
We have 
\[
\langle\mathbb{A}^{\mrm{out}}_T\u,\v\rangle=\int_{\Om}\eps^{-1}\boldsymbol{\zeta_{\u}}\cdot\curl\overline{\v}\,dx,\qquad\forall\u\in\mZ^{\mrm{out}}_T(\mu),\,\forall\v\in\mZ^{\beta}_T(\mu).
\]
Let us construct a continuous operator $\mathbb{T}:\mZ^{\beta}_T(\mu)\to\mZ^{\mrm{out}}_T(\mu)$ such that 
\begin{equation}\label{RelationIdentity}
\langle \mathbb{A}^{\mrm{out}}_T\,\mathbb{T}\u,\v\rangle=\int_{\Om}r^{2\beta}\curl\u\cdot\curl\overline{\v}\,dx,\qquad\forall\u,\,\v\in\mZ^{\beta}_T(\mu).
\end{equation}
Let $\u$ be an element of $\mZ^{\beta}_T(\mu)$. Then the field $r^{2\beta}\eps\,\curl\u$ belongs to $\boldsymbol{\mV}^0_{-\beta}(\Om)$. Since $A^{\mrm{out}}_{\eps}:\mathring{\mV}^{\mrm{out}}\to(\mathring{\mV}^1_{\beta}(\Om))^{\ast}$ is an isomorphism, there is a unique $\varphi=\alpha\,s^++\tilde{\varphi}\in\mathring{\mV}^{\mrm{out}}$ such that $A^{\mrm{out}}_{\eps}\varphi=-\div(r^{2\beta}\eps\,\curl\u)$. Observing that $\w:=r^{2\beta}\curl\u-\nabla\varphi\in\boldsymbol{\mV}^0_{\beta}(\Om)$ is such that $\div\,\w=0$ in $\Om$, according to the result of Proposition \ref{propoPotentialWeight}, we know that there is a unique $\boldsymbol{\psi}\in\mZ^{\beta}_T(1)$ such that 
\[
\curl \boldsymbol{\psi}=\eps\,(r^{2\beta}\curl\u-\nabla\varphi).
\]
At this stage, we emphasize that in general $\nabla\varphi\in\boldsymbol{\mV}^0_{\beta}(\Om)\setminus\Lspace^2(\Om)$. This is the reason why we are obliged to establish Proposition \ref{propoPotentialWeight}. Since $\boldsymbol{\psi}$ is in $\Lspace^2(\Om)$, when $A_{\mu}:\mH^1_{\#}(\Om)\to(\mH^1_{\#}(\Om))^{\ast}$ is an isomorphism, there is a unique $\phi\in\mH^1_{\#}(\Om)$ such that
\[
\int_{\Om} \mu\nabla\phi\cdot\nabla\overline{\phi'}\,dx=\int_{\Om} \mu\boldsymbol{\psi}\cdot\nabla\overline{\phi'}\,dx,\qquad\forall\phi'\in\mH^1_{\#}(\Om).
\]
Finally, we set $\mathbb{T}\u=\boldsymbol{\psi}-\nabla\phi$. It can be easily checked that this defines a continuous operator $\mathbb{T}:\mZ^{\beta}_T(\mu)\to\mZ^{\mrm{out}}_T(\mu)$. Moreover we have 
\[
\curl\mathbb{T}\u=\alpha\,\eps\nabla s^++\boldsymbol{\zeta_{\mathbb{T}\u}}\qquad\mbox{ with }\boldsymbol{\zeta_{\mathbb{T}\u}}=\eps\,(r^{2\beta}\curl\u-\nabla\tilde{\varphi}).
\]
As a consequence, indeed we have identity (\ref{RelationIdentity}). From Lemma \ref{LemmaNormeEquivHPlus}, we deduce that $\mathbb{A}^{\mrm{out}}_T\,\mathbb{T}:\mZ^{\beta}_T(\mu)\to(\mZ^{\beta}_T(\mu))^{\ast}$ is an isomorphism, and so that $\mathbb{A}^{\mrm{out}}_T$ is onto. 
It remains to show that $\mathbb{A}^{\mrm{out}}_T$ is injective.\\
\newline
If $\u\in\mZ^{\mrm{out}}_T(\mu)$ is in the kernel of $\mathbb{A}^{\mrm{out}}_T$, we have
$
\langle\mathbb{A}^{\mrm{out}}_T\u,\v\rangle=0
$
for all $\v\in\mZ^{\beta}_T(\mu)$. In particular from (\ref{eq-intH-sym}), we can write
\[
\langle\mathbb{A}^{\mrm{out}}_T\u,\u\rangle=\int_{\Om}\eps^{-1}|\boldsymbol{\zeta_{\u}}|^2\,dx+|c_{\u}|^2\int_{\Om}\div(\eps\nabla s^{+}) \overline{s^{+}}\,dx=0.
\]
Taking the imaginary part of the above identity, we obtain $c_{\u}=0$ (see the details in the proof of Proposition \ref{trappedmodes-H}). We deduce that $\u$ belongs to $\mZ^{-\beta}_T(\mu)$ and from (\ref{eq-intH-sym}), we infer that  $\langle\mathbb{A}^{\mrm{out}}_T\u,\mathbb{T}\u\rangle=\overline{\langle\mathbb{A}^{\mrm{out}}_T\,\mathbb{T}\u,\u\rangle}$. This gives
\[
0=\int_{\Om}r^{2\beta}|\curl\u|^2\,dx=0
\]
and shows that $\u=0$.
\end{proof}
\begin{proposition}\label{PropCompactH}
Under Assumptions 1--3, the embedding of the space $\mZ^{\mrm{out}}_T(\mu)$ in $\Lspace^2(\Om)$ is compact. As a consequence, the operator $\mathbb{K}^{\mrm{out}}_T:\mZ^{\mrm{out}}_T(\mu)\to(\mZ^{\beta}_T(\mu))^{\ast}$ defined in (\ref{DefOpChampH}) is compact. 
\end{proposition}
\begin{proof}
Let $(\u_n)$ be a sequence of elements of $\mZ^{\mrm{out}}_T(\mu)$ which is bounded. For all $n\in\N$, we have $\curl\u_n=c_{\u_n}\eps\nabla s^++\boldsymbol{\zeta}_{\u_n}$. By definition of the norm of $\mZ^{\mrm{out}}_T(\mu)$, the sequence $(c_{\u_n})$ is bounded in $\Cplx$. Let $\w$ be an element of $\mZ^{\mrm{out}}_T(\mu)$ such that $c_{\w}=1$ (if such $\w$ did not exist, then we would have $\mZ^{\mrm{out}}_T(\mu)=\mZ^{-\beta}_T(\mu)\subset\mX_T(\mu)$ and the proof would be even simpler). The sequence $(\u_n-c_{\u_n}\w)$ is bounded in $\mX_T(\mu)$. Since this space is compactly embedded in $\Lspace^2(\Om)$ when $A_{\mu}:\mH^1_{\#}(\Om)\to(\mH^1_{\#}(\Om))^{\ast}$ is an isomorphism (see \cite[Theorem 5.3]{BoCC14}), we infer we can extract from $(\u_n-c_{\u_n}\w)$ a subsequence which converges in $\Lspace^2(\Om)$. Since clearly we can also extract a subsequence of $(c_{\u_n})$ which converges in $\Cplx$, this shows that we can extract from $(\u_n)$ a subsequence which converges in $\Lspace^2(\Om)$. This shows that the embedding of $\mZ^{\mrm{out}}_T(\mu)$ in $\Lspace^2(\Om)$ is compact.\\
Now observing that for all $\u\in\mZ^{\mrm{out}}_T(\mu)$, we have
\[
\|\mathbb{K}^{\mrm{out}}_T\u\|_{(\mZ^{\beta}_T(\mu))^{\ast}} \le C\,\|\u\|_{\Om},
\]
we deduce that $\mathbb{K}^{\mrm{out}}_T:\mZ^{\mrm{out}}_T(\mu)\to(\mZ^{\beta}_T(\mu))^{\ast}$ is a compact operator.
\end{proof}
\noindent We can now state the main theorem of the analysis of the problem for the magnetic field.
\begin{theorem}\label{MainThmH}
Under Assumptions 1--3, for all $\om\in\R$ the operator $\mathbb{A}^{\mrm{out}}_T-\om^2\mathbb{K}^{\mrm{out}}_T:\mZ^{\mrm{out}}_T(\mu)\to(\mZ^{\beta}_T(\mu))^{\ast}$ is Fredholm of index zero. 
\end{theorem}
\begin{proof}
Since $\mathbb{K}^{\mrm{out}}_T:\mZ^{\mrm{out}}_T(\mu)\to(\mZ^{\beta}_T(\mu))^{\ast}$ is compact (Proposition \ref{PropCompactH}) and $\mathbb{A}^{\mrm{out}}_T:\mZ^{\mrm{out}}_T(\mu)\to(\mZ^{\beta}_T(\mu))^{\ast}$ is an isomorphism (Proposition \ref{propoUniquenessH}), $\mathbb{A}^{\mrm{out}}_T-\om^2\mathbb{K}^{\mrm{out}}_T:\mZ^{\mrm{out}}_N\to(\mZ^{\beta}_T(\mu))^{\ast}$ is Fredholm of index zero. 
\end{proof}
\noindent Finally we establish a result similar to Proposition \ref{trappedmodes} by using the formulation for the magnetic field.
\begin{proposition}\label{trappedmodes-H}
Under Assumptions 1 and 3, if $\u\in \mZ^{\mrm{out}}_T(\mu)$ is a solution of (\ref{MainPbH}) for $\J=0$, then $\u\in\mZ^{-\gamma}_T(\mu)\subset\mX_T(\mu)$ for all $\gamma$ satisfying (\ref{betainfbeta_0et1demi}).
\end{proposition}
\begin{proof}
Assume that $\u\in\mZ^{\mrm{out}}_T(\mu)$  satisfies
\[
\fintH_{\Om}\eps^{-1}\curl\u\cdot\curl\overline{\v}\,dx -\om^2 \int_\Om \mu \u\cdot\overline{\v}=0,\qquad \forall \v\in \mZ^{\beta}_T(\mu).
\]
Taking the imaginary part of this identity for $\v=\u$, since $\om$ is real, we get
\[
\Im m\,\left(\fintH_{\Om}\eps^{-1}\curl\u\cdot\curl\overline{\u}\,dx\right)=0.
\]
If $\curl\u=c_{\u}\,\eps\nabla s^++\boldsymbol{\zeta_{\u}}$ with $c_{\u}\in \Cplx$ and $\boldsymbol{\zeta_{\u}}\in\boldsymbol{\mV}^0_{-\beta}(\Om)$, 
according to (\ref{eq-intH-sym}), this can be written as 
\[
|c_{\u}|^2	\Im m\,\left(\int_{\Om}\div(\eps\nabla s^+ )\,\overline{s^{+}}\,dx\right) =0.
\]
Then one concludes as in the proof of Proposition \ref{trappedmodes} that $c_{\u}=0$, so that $\curl\u\in\boldsymbol{\mV}^0_{-\beta}(\Om)$. Therefore we have $\eps^{-1}\curl\u\in\mX_N(\eps)\subset\mX^{\mrm{out}}_N(\eps)$. From Lemma \ref{LemmaNormeEquiv}, we deduce that $\eps^{-1}\curl\u\in\boldsymbol{\mV}^0_{-\gamma}(\Om)$ for all $\gamma$ satisfying (\ref{betainfbeta_0et1demi}). This shows that $\u\in\mZ^{-\gamma}_T(\mu)$ for all $\gamma$ satisfying (\ref{betainfbeta_0et1demi}).
\end{proof}
\begin{remark}\label{frameworkIndbeta}
Assume that $\J\in\boldsymbol{\mV}^0_{-\gamma}(\Om)$ for all $\gamma$ satisfying (\ref{betainfbeta_0et1demi}). Assume also that zero is the only solution of (\ref{MainPbH}) with $\J=0$ for a certain $\beta_0$ satisfying (\ref{betainfbeta_0et1demi}). Then Theorem \ref{MainThmH} and Proposition \ref{trappedmodes-H} guarantee that (\ref{MainPbH}) is well-posed for all $\gamma$ satisfying (\ref{betainfbeta_0et1demi}). Moreover Proposition \ref{trappedmodes-H} allows one to show that all the solutions of (\ref{MainPbH}) for $\gamma$ satisfying (\ref{betainfbeta_0et1demi}) coincide.
\end{remark}

\subsection{Analysis in the classical framework}
In the previous paragraph, we proved that the formulation (\ref{MainPbH}) for the magnetic field with a solution in $\mZ^{\mrm{out}}_T(\mu)$ and test functions in $\mZ^{\beta}_T(\mu)$ is well-posed. Here, we study the properties of the problem for the magnetic field set in the classical space $\mX_T(\mu)$. More precisely, we consider the problem
\begin{equation}\label{MainPbHCla}
\begin{array}{|l}
\mbox{Find }\u\in\mX_T(\mu)\mbox{ such that }\\
\dsp\int_{\Om}\eps^{-1}\curl\u\cdot\curl\overline{\v}\,dx -\om^2 \int_\Om \mu \u\cdot\overline{\v}= \int_\Om \eps^{-1}\J\cdot \curl\overline{\v},\qquad \forall\v\in\mX_T(\mu).
\end{array}
\end{equation}
Working as in the proof of Lemma \ref{EquivHcurlH}, one shows that under Assumptions 1, 2, the field $\u$ is a solution of (\ref{MainPbHCla}) if and only if it solves the problem 
\begin{equation}\label{MainPbHClaHcurl}
\begin{array}{|l}
\mbox{Find }\u\in\Hspace(\curl)\mbox{ such that }\\
\dsp\int_{\Om}\eps^{-1}\curl\u\cdot\curl\overline{\v}\,dx -\om^2 \int_\Om \mu \u\cdot\overline{\v}= \int_\Om \eps^{-1}\J\cdot \curl\overline{\v},\qquad \forall\v\in\Hspace(\curl).
\end{array}
\end{equation}
\noindent Define the continuous operators $\mathbb{A}_T:\mX_T(\mu)\to(\mX_T(\mu))^{\ast}$ and $\mathbb{K}_T:\mX_T(\mu)\to(\mX_T(\mu))^{\ast}$ such that for all $\u\in\mX_T(\mu)$, $\v\in\mX_T(\mu)$,
\[
\langle\mathbb{A}_T\u,\v\rangle = \int_{\Om}\eps^{-1}\curl\u\cdot\curl\overline{\v}\,dx,\qquad\langle\mathbb{K}_T\u,\v\rangle = \int_{\Om}\mu\u\cdot\overline{\v}\,dx.
\]
As for $\mathbb{A}_N$ and $\mathbb{K}_N$, we emphasize that these are the classical operators which appear in the analysis of the magnetic field, for example when $\eps$ and $\mu$ are positive in $\Om$.
\begin{proposition}\label{propoUniquenessCla}
Under Assumptions 1--3, for all $\om\in\mathbb{C}$ the operator $\mathbb{A}_T-\om^2\mathbb{K}_T:\mX_T(\mu)\to(\mX_T(\mu))^{\ast}$ is not Fredholm. 
\end{proposition}
\begin{proof}
From \cite[Theorem 5.3 and Corollary 5.4]{BoCC14}, we know that under the Assumptions 1, 2, the embedding of $\mX_T(\mu)$ in $\Lspace^2(\Om)$ is compact. This allows us to prove that $\mathbb{K}_T:\mX_T(\mu)\to(\mX_T(\mu))^{\ast}$ is a compact operator. Therefore, it suffices to show that $\mathbb{A}_T:\mX_T(\mu)\to(\mX_T(\mu))^{\ast}$ is not Fredholm. Let us work by contraction assuming that $\mathbb{A}_T$ is Fredholm. Since this operator is self-adjoint (it is symmetric and bounded), necessarily it is of index zero.\\ 
\newline
$\star$ If $\mathbb{A}_T$ is injective, then it is an isomorphism. Let us show that in this case, $A_{\eps}:\mH^1_0(\Om)\to(\mH^1_0(\Om))^{\ast}$ is an isomorphism (which is not the case by assumption). To proceed, we construct a continuous operator $\texttt{T}:\mH^1_0(\Om)\to\mH^1_0(\Om)$ such that 
\begin{equation}\label{IdentityIsom}
\langle A_{\eps}\varphi,\texttt{T}\varphi'\rangle=\int_{\Om}\eps\nabla\varphi\cdot\nabla(\overline {\texttt{T}\varphi'})\,dx=\int_{\Om}\nabla\varphi\cdot\nabla\overline {\varphi'}\,dx,\qquad \forall \varphi,\varphi'\in\mH^1_0(\Om).
\end{equation}
When $\mathbb{A}_T$ is an isomorphism, for any $\varphi'\in\mH^1_0(\Om)$, there is a unique $\boldsymbol{\psi}\in\mX_T(\mu)$ such that
\[
\int_{\Om}\eps^{-1}\curl\boldsymbol{\psi}\cdot\curl\overline{\boldsymbol{\psi}'}\,dx=\int_{\Om}\eps^{-1}\nabla\varphi'\cdot\curl\overline{\boldsymbol{\psi}'}\,dx,\qquad\forall \boldsymbol{\psi}'\in\mX_T(\mu).
\]
Using item $iii)$ of Proposition \ref{propoPotential}, one can show that there is a unique $\texttt{T}\varphi'\in\mH^1_0(\Om)$ such that
\[
\nabla(\texttt{T}\varphi')=\eps^{-1}(\nabla\varphi'-\curl\boldsymbol{\psi}).
\]
This defines our operator $\texttt{T}:\mH^1_0(\Om)\to\mH^1_0(\Om)$ and one can verify that it is continuous. Moreover, integrating by parts, we indeed get (\ref{IdentityIsom}) which guarantees, according to the Lax-Milgram theorem, that $A_{\eps}:\mH^1_0(\Om)\to\mH^1_0(\Om)$   is an isomorphism.\\
\newline
$\star$ If $\mathbb{A}_T$ is not injective, it has a kernel of finite dimension $N\ge1$ which coincides with $\mrm{span}(\boldsymbol{\lambda}_1,\dots,\boldsymbol{\lambda}_N)$, where $\boldsymbol{\lambda}_1,\dots,\boldsymbol{\lambda}_N\in\mX_T(\mu)$ are linearly independent functions such that $(\curl\boldsymbol{\lambda}_i,\curl\boldsymbol{\lambda}_j)_{\Om}=\delta_{ij}$ (the Kronecker symbol). Introduce the space
\[
\tilde{\mX}_T(\mu):=\{\u\in\mX_T(\mu)\,|\,(\curl\u,\curl\boldsymbol{\lambda}_i)_{\Om}=0,\,i=1,\dots N\}.
\]
as well as the operator $\tilde{\mathbb{A}}_T:\tilde{\mX}_T(\mu)\to\tilde{\mX}_T(\mu)$ such that
\[
\langle\tilde{\mathbb{A}}_T\u,\v\rangle = \int_{\Om}\eps^{-1}\curl\u\cdot\curl\overline{\v}\,dx,\qquad\forall\u,\v\in\tilde{\mX}_T(\mu).
\]
Then $\tilde{\mathbb{A}}_T$ is an isomorphism. Let us construct a new operator $\texttt{T}:\mH^1_0(\Om)\to\mH^1_0(\Om)$ to have something looking like (\ref{IdentityIsom}). For a given $\varphi'\in\mH^1_0(\Om)$, introduce $\boldsymbol{\psi}\in\tilde{\mX}_T(\mu)$ the function such that
\begin{equation}\label{LineSmallPb}
\int_{\Om}\eps^{-1}\curl\boldsymbol{\psi}\cdot\curl\overline{\boldsymbol{\psi}'}\,dx=\int_{\Om}(\eps^{-1} \nabla\varphi' -\sum_{i=1}^{N} \alpha_i \curl\boldsymbol{\lambda}_i)\cdot\curl\overline{\boldsymbol{\psi}'}\,dx,\qquad\forall \boldsymbol{\psi}'\in\tilde{\mX}_T(\mu),
\end{equation}
where for $i=1,\dots,N$, we have set $\alpha_i:=\textstyle\int_\Omega \eps^{-1}\nabla \varphi'\cdot\curl\boldsymbol{\lambda}_i\,dx$. Observing that (\ref{LineSmallPb}) is also valid for $\boldsymbol{\psi}'=\boldsymbol{\lambda}_i$, $i=1,\dots,N$, we infer that there holds 
\[
\int_{\Om}\eps^{-1}\curl\boldsymbol{\psi}\cdot\curl\overline{\boldsymbol{\psi}'}\,dx=\int_{\Om}(\eps^{-1} \nabla\varphi' -\sum_{i=1}^{N} \alpha_i \curl\boldsymbol{\lambda}_i)\cdot\curl\overline{\boldsymbol{\psi}'}\,dx,\qquad\forall \boldsymbol{\psi}'\in\mX_T(\mu).
\]
Using again item $iii)$ of Proposition \ref{propoPotential}, we deduce that there is a unique $\texttt{T}\varphi'\in\mH^1_0(\Om)$ such that
\[
\nabla(\texttt{T}\varphi')=\eps^{-1}(\nabla\varphi'-\curl\boldsymbol{\psi})-\sum_{i=1}^{N} \alpha_i \curl\boldsymbol{\lambda}_i.
\] 
This defines the new continuous operator $\texttt{T}:\mH^1_0(\Om)\to\mH^1_0(\Om)$. Then one finds 
\[
\langle A_{\eps}\varphi,\texttt{T}\varphi'\rangle=\int_{\Om}\eps\nabla\varphi\cdot\nabla(\overline {\texttt{T}\varphi'})\,dx=\int_{\Om}\nabla\varphi\cdot\nabla\overline {\varphi'}\,dx-\sum_{i=1}^{N}\overline{\alpha_i}\int_\Om  \eps\nabla\varphi\cdot\curl\overline {\lambda_i}\,dx ,\quad \forall \varphi,\varphi'\in\mH^1_0(\Om).
\]
This shows that $\texttt{T}$ is a left parametrix for the self adjoint operator  $A_{\eps}$. Therefore, $A_{\eps}:\mH^1_0(\Om)\to\mH^1_0(\Om)$ is Fredholm of index zero. Note that then, one can verify that $\dim\ker\,A_{\eps}=\dim\ker\,\mathbb{A}_T$. And more precisely, we have $\ker\,A_{\eps}=\mrm{span}(\gamma_1,\dots,\gamma_N)$ where $\gamma_i\in\mH^1_0(\Om)$ is the function such that
\[
\nabla \gamma_i=\eps^{-1}\curl\boldsymbol{\lambda}_i
\]
(existence and uniqueness of $\gamma_i$ is again a consequence of item $iii)$ of Proposition \ref{propoPotential}). But by assumption, $A_{\eps}$ is not a Fredholm operator. This ends the proof by contradiction.
\end{proof}
\begin{remark}
In the article \cite{BoCC14}, it is proved that if $A_{\eps}:\mH^1_0(\Om)\to\mH^1_0(\Om)$  is an isomorphism (resp. a Fredholm operator of index zero), then $\mathbb{A}_T:\mX_T(1)\to(\mX_T(1))^{\ast}$ is an isomorphism (resp. a  Fredholm operator of index zero). Here we have established the converse statement.
\end{remark}
\begin{remark}
We emphasize that the problems (\ref{MainPbECla}) for the electric field and (\ref{MainPbHCla}) for the magnetic in the usual spaces $\mX_N(\eps)$ and $\mX_T(\mu)$ have different properties. Problem (\ref{MainPbECla}) is well-posed but is not equivalent to the corresponding  problem in $\Hspace_N(\curl)$, so that its solution in general is not a distributional solution of Maxwell's equations. On the contrary, problem (\ref{MainPbHCla}) is equivalent to problem (\ref{MainPbHClaHcurl}) in $\Hspace(\curl)$ but it is not well-posed.
\end{remark}

\subsection{Expression of the singular coefficient}
Under Assumptions 1--3, Theorem \ref{MainThmH} guarantees that for all $\om\in\R$ the operator $\mathbb{A}^{\mrm{out}}_T-\om^2\mathbb{K}^{\mrm{out}}_T:\mZ^{\mrm{out}}_T(\mu)\to(\mZ^{\beta}_T(\mu))^{\ast}$ is Fredholm of index zero. Assuming that it is injective, the problem (\ref{MainPbH}) admits a unique solution $\u$ with $\curl\u=c_{\u}\,\eps\nabla s^++\boldsymbol{\zeta}_{\u}$. As in \S\ref{paraSingCoef}, the goal of this paragraph is to derive a formula for the coefficient $c_{\u}$ which does not require to know $\u$.\\
\newline
For $\om\in\R$, introduce the field $\w_T\in\mZ^{\beta}_T(\mu)$ such that 
\begin{equation}\label{defwSinguBis}
\dsp\int_{\Om}\eps^{-1}\boldsymbol{\zeta_{\v}}\cdot\curl\overline{\w_T}\,dx -\om^2 \int_\Om \mu \v\cdot \overline{\w_T}\,dx
		= \int_\Om \zeta_{\v}\cdot\nabla\overline{s^+}\,dx,\qquad \forall\v\in\mZ^{\mrm{out}}_T(\mu).
\end{equation}
Note that $\w_T$ is well-defined because $(\mathbb{A}^{\mrm{out}}_T-\om^2\mathbb{K}^{\mrm{out}}_T)^{\ast}:\mZ^{\beta}_T(\mu)\to(\mZ^{\mrm{out}}_T(\mu))^{\ast}$ is an isomorphism. 
\begin{theorem}
Assume that $\om\in\R$, Assumptions 1--3 are valid and $\mathbb{A}^{\mrm{out}}_T-\om^2\mathbb{K}^{\mrm{out}}_T:\mZ^{\mrm{out}}_T(\mu)\to(\mZ^{\beta}_T(\mu))^{\ast}$ is injective. Let $\u$ denote the solution of the magnetic problem (\ref{MainPbH}). Then the coefficient $c_{\u}$ in the decomposition $\curl\u=c_{\u}\,\eps\nabla s^++\boldsymbol{\zeta}_{\u}$ is given by the formula
\begin{equation}\label{FormulaCoefSinguH}
c_{\u}=i\om\int_{\Om} \eps^{-1}\J\cdot\curl\overline{\w_T}\,dx\bigg/\int_{\Om}\div(\eps\nabla s^{+})\,\overline{s^+}\,dx.
\end{equation}
Here $\w_T$ is the function which solves (\ref{defwSinguBis}).
\end{theorem}
\begin{proof}
By definition of $\u$, we have
\[
\dsp\int_{\Om}\eps^{-1}\boldsymbol{\zeta_{\u}}\cdot\curl\overline{\w_T}\,dx -\om^2 \int_\Om \mu \u\cdot \overline{\w_T}\,dx= i\om\int_{\Om} \eps^{-1}\J\cdot\curl\overline{\w_T}\,dx.
\]
On the other hand, from (\ref{defwSinguBis}), we can write
\[
\dsp\int_{\Om}\eps^{-1}\boldsymbol{\zeta_{\u}}\cdot\curl\overline{\w_T}\,dx -\om^2 \int_\Om \mu \u\cdot \overline{\w_T}\,dx
		= \int_\Om \boldsymbol{\zeta_{\u}}\cdot\nabla\overline{s^+}\,dx.
\]
From these two relations, using (\ref{eq-intH-sym}), we deduce that 
\[
i\om\int_{\Om} \eps^{-1}\J\cdot\curl\overline{\w_T}\,dx=\int_\Om \boldsymbol{\zeta_{\u}}\cdot\nabla\overline{s^+}\,dx=c_{\u}\int_{\Om}\div(\eps\nabla s^{+})\,\overline{s^+}\,dx.
\]
This gives (\ref{FormulaCoefSinguH}).
\end{proof}

\section{Conclusion}\label{SectionConclusion}
In this work, we studied the Maxwell's equations in presence of hypersingularities for the scalar problem involving $\eps$. We considered both the problem for the electric field and for the magnetic field. Quite naturally, in order to obtain a framework where well-posedness holds, it is necessary to modify the spaces in different ways. More precisely,  we changed the condition on the field itself for the electric problem and on the curl of the field for the magnetic problem. A noteworthy difference in the analysis of the two problems is that for the electric field, we are led to work in a Hilbertian framework, whereas for the magnetic field we have not been able to do so.\\
\newline
Of course, we could have assumed that the scalar problem involving $\eps$ is well-posed in $\mH^1_0(\Om)$ and that hypersingularities exist for the problem in $\mu$. This would have been similar mathematically. Physically, however, this situation seems to be a bit less relevant because it is harder to produce negative $\mu$ without dissipation. We assumed that the domain $\Om$ is simply connected and that $\partial\Om$ is connected. When these assumptions are not met, it is necessary to adapt the analysis (see \S8.2 of \cite{BoCC14} for the study in the case where the scalar problems are well-posed in the usual $\mH^1$ framework). This has to be done. Moreover, for the conical tip, at least numerically, one finds that several singularities can exist (see the calculations in \cite{KCHWS14}). In this case, the analysis should follow the same lines but this has to be written. On the other hand, in this work, we focused our attention on a situation where the interface between the positive and the negative material has a conical tip. It would be interesting to study a setting where there is a wedge instead. In this case, roughly speaking, one should deal with a continuum of singularities. We have to mention that the analysis of the scalar problems for a wedge of negative material in the non standard framework has not been done. Finally, considering a conical tip with both critical $\eps$ and $\mu$ is a direction that we are investigating.  

\appendix
\section{Vector potentials, part 1}
\begin{proposition}\label{propoPotential}
Under Assumption 1, the following assertions hold.\\
\newline
i)\, According to \cite[Theorem 3.12]{AmrBerDau98}, if $\u\in \Lspace^2(\Om)$ satisfies $\div\,\u=0$ in $\Om$, then there exists a unique $\boldsymbol{\psi}\in \mX_T(1)$ such that $\u= \curl \boldsymbol{\psi}$.\\
\newline
ii)\, According to \cite[Theorem 3.17]{AmrBerDau98}), if $\u\in \Lspace^2(\Om)$ satisfies $\div\,\u=0$ in $\Omega$ and  $\u\cdot\nu=0$ on $\partial\Omega$, then there exists a unique $\boldsymbol{\psi}\in \mX_N(1)$ such that $\u= \curl \boldsymbol{\psi}$.\\
\newline
iii)\, If $\u\in \Lspace^2(\Om)$ satisfies $\curl\u=0$ in $\Omega$ and $\u\times \nu=0$ on $\partial\Omega$, then there exists (see \cite[Theorem 3.41]{Mon03}) a unique $p\in \mH^1_0(\Omega)$ such that $\u= \nabla p$.\\
\newline
iv)\, Every $\u\in \Lspace^2(\Om)$ can be decomposed as follows (\cite[Theorem 3.45]{Mon03})
$$
\u= \nabla p +\curl \boldsymbol{\psi},
$$
with $p\in \mH^1_0(\Omega)$ and $\boldsymbol{\psi}\in \mX_T(1)$ which are uniquely defined. \\
\newline
v)\, Every $\u\in \Lspace^2(\Om)$ can be decomposed as follows (\cite[Remark 3.46]{Mon03})
$$
\u= \nabla p +\curl \boldsymbol{\psi},
$$
with $p\in \mH^1_{\#}(\Om)$ and $\boldsymbol{\psi}\in \mX_N(1)$ which are uniquely defined. 
\end{proposition}

\begin{proposition}\label{propoLaplacienVect}
Under Assumption 1, if $\boldsymbol{\psi}$ satisfies one of the following conditions
\newline 
i) $\boldsymbol{\psi}\in\mX_N(1)$ and $\boldsymbol{\Delta \psi}\in \Lspace^2(\Omega)$,
\newline 
ii) $\boldsymbol{\psi}\in\mX_T(1)$, $\curl\boldsymbol{\psi}\times \nu=0$ on $\partial\Omega$ and $\boldsymbol{\Delta \psi}\in \Lspace^2(\Omega)$, \newline then for all $\beta<1/2$,  we have $\curl\boldsymbol{\psi}\in\boldsymbol{\mV}^0_{-\beta}(\Om)$ and there is a constant $C>0$ independent of $\boldsymbol{\psi}$ such that
\begin{equation}\label{EstimateWeightedLap}
\|\curl\boldsymbol{\psi}\|_{\boldsymbol{\mV}^0_{-\beta}(\Om)} \le C\,\|\boldsymbol{\Delta \psi}\|_{\Om}.
\end{equation}
\end{proposition}
\begin{proof}
It suffices to prove the result for $\beta\in(0;1/2)$. Let $\boldsymbol{\psi}\in\mX_N(1)\cup \mX_T(1)$. Since $\curl\curl\boldsymbol{\psi}=-\boldsymbol{\Delta \psi} $, integrating by parts we get    
\[
\|\curl\boldsymbol{\psi}\|^2_{\Om}=-\int_{\Om}\boldsymbol{\Delta \psi}\cdot\overline{\boldsymbol{\psi}}\,dx.
\]
Note that the boundary term vanishes because either  $\boldsymbol{\psi}\times \nu=0$ or  $\curl\boldsymbol{\psi}\times \nu=0$ on $\partial\Omega$. 
This furnishes the estimate 
\begin{equation}
\|\curl\boldsymbol{\psi}\|_{\Om} \le C\,\|\boldsymbol{\Delta \psi}\|_{\Om}.\label{nomrcurlL2}
\end{equation}  Now working with cut-off functions, we refine the estimate at the origin to get (\ref{EstimateWeightedLap}). 
\newline
	Let us consider a smooth cut-off function $\chi$, compactly supported in $\Omega$,  equal to one in a neighbourhood of $O$.
	To prove the proposition,  it suffices in addition to (\ref{nomrcurlL2}) to prove that $\curl(\chi\boldsymbol{\psi})\in\boldsymbol{\mV}^0_{-\beta}(\Om)$  together with the following estimate    $\|\curl (\chi\boldsymbol{\psi})\|_{\boldsymbol{\mV}^0_{-\beta}(\Om)} \le C\,\|\boldsymbol{\Delta \psi}\|_{\Om}$. \\
	First of all,  since $\curl (\chi\boldsymbol{\psi})\in \Lspace^2(\Omega)$ and $\div(\chi\boldsymbol{\psi})=\nabla\chi\cdot\boldsymbol{\psi}\in \mL^2(\Omega)$, we know that  $\chi\boldsymbol{\psi_i}\in\mH^1_0(\Omega)$ for $i=1,2,3$ and we  have  
\[
	\|\curl (\chi\boldsymbol{\psi})\|_{\Omega}^2+\|\div (\chi\boldsymbol{\psi})\|_{\Omega}^2=\sum_{i=1}^3\|\nabla(\chi\boldsymbol{\psi_i})\|_{\Omega}^2.
\]
	From the previous identity, (\ref{nomrcurlL2}) and  Proposition \ref{PropoEmbeddingCla}, we deduce 
\begin{equation}\label{normH1}
	\left(\|\boldsymbol{\psi}\|^2_{\Om}+\sum_{i=1}^3\|\nabla (\chi\boldsymbol{\psi_i})\|_{\Omega}^2\right)^{1/2}\leq  C\,\|\boldsymbol{\Delta \psi}\|_{\Om}. 
\end{equation} 
Note that, \eqref{normH1} is also valid if we replace $\chi$ by any other smooth function with compact support in $\Om$. Now setting $f_i=\Delta(\chi \boldsymbol{\psi_i})$ for $i=1,2,3$, we have
\begin{equation}	\label{laplaciantranc}
	f_i= \chi {\Delta} \boldsymbol{\psi}_i+2\,\nabla \chi\cdot\nabla\boldsymbol{\psi}_i +\boldsymbol{\psi}_i  {\Delta} \chi . 
	\end{equation} 
	By writing that $\nabla \chi\cdot\nabla\boldsymbol{\psi}_i=\div(\boldsymbol{\psi}_i\nabla \chi)-\boldsymbol{\psi}_i\Delta\chi$ and replacing  $\chi$ by $\partial_j\chi$ in 
	(\ref{normH1}) for $j=1,2,3$,  we deduce that for $i=1,2,3$, $f_i$ belongs to $\mL^2(\Omega)$ and satisfies
\[
\|f_i\|_{ \Om}\leq C \|\boldsymbol{\Delta \psi}\|_{\Om}.
\]
Note that since $\beta\in(0;1/2)$, we have $\mathring{\mV}^1_\beta(\Om)\subset\mV^0_{\beta-1}\subset\mL^2(\Om)$ and so $\mL^2(\Om)\subset(\mathring{\mV}^1_\beta(\Om))^\ast$. Now starting from the fact that $\chi \boldsymbol{\psi_i}\in  \mH^1_0(\Om)$ in addition to  $\Delta(\chi \boldsymbol{\psi_i})=f_i\in\mL^2(\Om)\subset(\mathring{\mV}^1_\beta(\Om))^\ast$, by applying  Proposition \ref{PropoLaplaceOp}, we deduce that $\chi \boldsymbol{\psi_i}\in\mathring{\mV}^1_{-\beta}(\Om)$ with the estimate
\[	
	\|\chi \boldsymbol{\psi_i} \|_{\mathring{\mV}^1_{-\beta}(\Om)}\leq C\, \|f_i\|_{(\mathring{\mV}^1_\beta(\Om))^\ast}\leq C\, \|f_i\|_\Om.
	\]
	As a consequence, $\curl(\chi\boldsymbol{\psi})\in\boldsymbol{\mV}^0_{-\beta}(\Om)$ and 
	\[\|\curl(\chi\boldsymbol{\psi})\|_{\boldsymbol{\mV}^0_{-\beta}(\Om)}\leq  C \sum_{i=1}^3 \|\chi \boldsymbol{\psi_i} \|_{\mathring{\mV}^1_{-\beta}(\Om)}\leq \sum_{i=1}^3  \|f_i\|_\Om \leq C \|\boldsymbol{\Delta \psi}\|_{\Om}, \]
	which concludes the proof.
\end{proof}
\begin{proposition}\label{propoLaplacienVectcompact}
Under Assumption 1, the following assertions hold:\\
i)	if $(\boldsymbol{\psi_n})$ is a bounded sequence of elements of $\mX_N(1)$  such that $(\boldsymbol{\Delta \psi_n})$ is bounded in $\Lspace^2(\Omega)$, then one can extract a subsequence such that $(\curl\boldsymbol{\psi_n})$ converges in $\boldsymbol{\mV}^0_{-\beta}(\Om)$ for all $\beta\in(0;1/2)$;
	\newline
ii) if $(\boldsymbol{\psi_n})$ is a bounded sequence of elements of $\mX_T(1)$  such that  $\curl\boldsymbol{\psi_n}\times \nu=0$ on $\partial\Omega$ and such that $(\boldsymbol{\Delta \psi_n})$ is bounded in $\Lspace^2(\Omega)$, then one can extract a subsequence such that $(\curl\boldsymbol{\psi_n})$ converges in $\boldsymbol{\mV}^0_{-\beta}(\Om)$  for all $\beta\in(0;1/2)$.  
\end{proposition}
\begin{proof}
Let us establish the first assertion, the proof of the second one being similar. Let $(\boldsymbol{\psi_n})$ be a bounded sequence of elements of $\mX_N(1)$ such that $(\boldsymbol{\Delta \psi_n})$ is bounded in $\Lspace^2(\Omega)$. Observing that $\curl\curl\boldsymbol{\psi_n}=-\boldsymbol{\Delta \psi_n} $, we deduce that $(\curl\boldsymbol{\psi_n})$ is a bounded sequence of $\mX_T(1)$. Since the spaces $\mX_N(1)$ and $\mX_T(1)$ are compactly  embedded in  $\Lspace^2(\Omega)$ (see Proposition \ref{PropoEmbeddingCla}), one can extract a subsequence such that both $ (\boldsymbol{\psi_n})$ and $(\curl\boldsymbol{\psi_n})$ converge in $\Lspace^2(\Omega)$. 
	\newline
	Then, working as in the proof of Proposition \ref{propoLaplacienVect}, we can show that for a smooth cut-off function  $\chi$  compactly supported in $\Omega$ and  equal to one in a neighbourhood of $O$, the sequence $(\chi\boldsymbol{\psi_n})$ is bounded in $ \boldsymbol{\mV}^2_{\gamma}(\Om):=(\mV^2_{\gamma}(\Om))^3$  for all $\gamma>1/2$. To obtain this result, we use in particular the fact that if $\mathscr{O}\subset\R^3$ is a smooth bounded domain such that $O\in\mathscr{O}$, then $\Delta:\mV^2_{\gamma}(\mathscr{O})\cap\mathring{\mV}^1_{\gamma-1}(\mathscr{O})\to\mV^0_{\gamma}(\mathscr{O})$ is an isomorphism for all $\gamma\in(1/2;3/2)$ (see \cite[\S1.6.2]{MaNP00}).	 Finally, to conclude to the result of the proposition, we use the fact $ \boldsymbol{\mV}^2_{\gamma}(\mathscr{O})$  is compactly embedded in $ \boldsymbol{\mV}^1_{\gamma'}(\mathscr{O})$ a soon as $\gamma-1<\gamma'$ (\cite[Lemma 6.2.1]{KoMR97}). This allows us to prove that for all $\beta<1/2$, the subsequence  $(\chi\boldsymbol{\psi_n})$ converges in $\boldsymbol{\mV}^1_{-\beta}(\Om)$, so that  $(\curl\boldsymbol{\psi_n})$ converges in $\boldsymbol{\mV}^0_{-\beta}(\Om)$.  
	\end{proof}

\noindent The next two lemmas are results of additional regularity for the elements of classical Maxwell's spaces that are direct consequences of Propositions \ref{propoLaplacienVect} and \ref{propoLaplacienVectcompact}.
\begin{lemma}\label{LemmaWeightedCla}
	Under Assumption 1,  for all $\beta\in(0;1/2)$,  $\mX_T(1)$ is compactly embedded in $\boldsymbol{\mV}^0_{-\beta}(\Om)$. In particular,  there is a constant $C>0$ such that
	\begin{equation}\label{EstimaWeightedCla}
	\|\u\|_{\boldsymbol{\mV}^0_{-\beta}(\Om)} \le C\,\|\curl\u\|_{\Om},\qquad\forall\u\in\mX_T(1).
	\end{equation}
\end{lemma}
\begin{proof}
Let $\u$ be an element of $\mX_T(1)$. From the item $ii)$ of Proposition \ref{propoPotential}, we know that there exists $\boldsymbol{\psi}\in \mX_N(1)$ such that $\u= \curl \boldsymbol{\psi}$. Using that $\boldsymbol{-\Delta\psi}= \curl\u \in \Lspace^2(\Omega)$, from Proposition \ref{propoLaplacienVect}, we get that $\u\in\boldsymbol{\mV}^0_{-\beta}(\Om)$ together with the estimate 
\[
\|\curl\boldsymbol{\psi}\|_{\boldsymbol{\mV}^0_{-\beta}(\Om)} \le C\,\|\curl\u\|_{\Om}.
\]
This gives (\ref{EstimaWeightedCla}). Now suppose that $(\u_n)$ is a bounded sequence of elements of $\mX_T(1)$. Then there exists a bounded sequence $(\boldsymbol{\psi}_n)$ of elements of $\mX_N(1)$ such that $\u_n= \curl \boldsymbol{\psi}_n$. Since $(\curl \u_n=-\Delta \boldsymbol{\psi}_n)$ is bounded in $\Lspace^2(\Omega)$, the first item of Proposition \ref{propoLaplacienVectcompact} implies that there is a subsequence such that $(\u_n)$ converges in  $\boldsymbol{\mV}^0_{-\beta}(\Om)$.
\end{proof}

\begin{lemma}\label{LemmaWeightedClaBis}
Under Assumption 1,  for all $\beta\in(0;1/2)$,  $\mX_N(1)$ is compactly embedded in $\boldsymbol{\mV}^0_{-\beta}(\Om)$. In particular,  there is a constant $C>0$ such that
\[
\|\u\|_{\boldsymbol{\mV}^0_{-\beta}(\Om)} \le C\,\|\curl\u\|_{\Om},\qquad\forall\u\in\mX_N(1).
\]
\end{lemma}
\begin{proof}
The proof is similar to the one of Lemma \ref{LemmaWeightedCla}.  
\end{proof}

\section{Vector potentials, part 2}

First we establish an intermediate lemma which can be seen as a result of well-posedness for Maxwell's equations in weighted spaces with $\eps=\mu=1$ in $\Om$. Define the continuous operator $\mathbb{B}_T:\mZ^{\beta}_T(1)\to(\mZ^{-\beta}_T(1))^{\ast}$ such that for all $\boldsymbol{\psi}\in\mZ^{\beta}_T(1)$, $\boldsymbol{\psi}'\in\mZ^{-\beta}_T(1)$, 
\[
\langle\mathbb{B}_T\boldsymbol{\psi},\boldsymbol{\psi}'\rangle = \int_{\Om}\curl\boldsymbol{\psi}\cdot\curl\overline{\boldsymbol{\psi}'}\,dx.
\]
\begin{lemma}\label{LemmaWPPoids}
Under Assumption 1, for $0\le\beta<1/2$, the operator $\mathbb{B}_T:\mZ^{\beta}_T(1)\to(\mZ^{-\beta}_T(1))^{\ast}$ is an isomorphism.
\end{lemma}
\begin{proof}
Let $\boldsymbol{\psi}$ be an element of $\mZ^{\beta}_T(1)$. According to Proposition \ref{PropoLaplaceOp}, there is a unique $\varphi\in\mathring{\mV}^1_{-\beta}(\Om)$ such that 
\[
\int_{\Om}\nabla\varphi\cdot\nabla\overline{\varphi'}\,dx=\int_{\Om}r^{2\beta}\curl\boldsymbol{\psi}\cdot\nabla\overline{\varphi'}\,dx,\qquad\forall \varphi'\in\mathring{\mV}^1_{\beta}(\Om).
\]
Then denote $\mathbb{T}\boldsymbol{\psi}\in\mZ^{-\beta}_T(1)$ the function such that
\[
\curl(\mathbb{T}\boldsymbol{\psi})=r^{2\beta}\curl\boldsymbol{\psi}-\nabla\varphi.
\]
Observe that $\mathbb{T}\boldsymbol{\psi}$ is well-defined according to the item $i)$ of Proposition \ref{propoPotential}. This defines a continuous operator $\mathbb{T}:\mZ^{\beta}_T(1)\to \mZ^{-\beta}_T(1)$. We have 
\[
\langle\mathbb{B}_T\boldsymbol{\psi},\mathbb{T}\boldsymbol{\psi}\rangle=\dsp\int_{\Om}\curl\boldsymbol{\psi}\cdot\overline{\curl(\mathbb{T}\boldsymbol{\psi})}\,dx=\|r^{\beta}\curl\boldsymbol{\psi}\|^2_{\Om}=\|\curl\boldsymbol{\psi}\|^2_{\boldsymbol{\mV}^0_{\beta}(\Om)}.
\]
Adapting the proof of Lemma \ref{LemmaNormeEquivHPlus}, one can show that $\|\curl\cdot\|_{\boldsymbol{\mV}^0_{\beta}(\Om)}$ is a norm which is equivalent to the natural norm of $\mZ^{\beta}_T(1)$. Therefore, from the Lax-Milgram theorem, we infer that $\mathbb{T}^{\ast}\mathbb{B}_T$ is an isomorphism which shows that $\mathbb{B}_T$ is injective and that its image is closed in $(\mZ^{-\beta}_T(1))^{\ast}$. And from that, we deduce that $\mathbb{B}_T$ is onto if and only if its adjoint is injective. The adjoint of $\mathbb{B}_T$ is the operator $\mathbb{B}_T^{\ast}:\mZ^{-\beta}_T(1)\to(\mZ^{\beta}_T(1))^{\ast}$ such that for all $\boldsymbol{\psi}\in\mZ^{-\beta}_T(1)$, $\boldsymbol{\psi}'\in\mZ^{\beta}_T(1)$, 
\begin{equation}\label{DefAdjoint}
\langle\mathbb{B}_T^{\ast}\boldsymbol{\psi},\boldsymbol{\psi}'\rangle = \int_{\Om}\curl\boldsymbol{\psi}\cdot\curl\overline{\boldsymbol{\psi}'}\,dx.
\end{equation}
If $\mathbb{B}_T^{\ast}\boldsymbol{\psi}=0$, then taking $\boldsymbol{\psi}'=\boldsymbol{\psi}\in\mZ^{-\beta}_T(1)\subset \mZ^{\beta}_T(1)$ in (\ref{DefAdjoint}), we obtain $\|\curl\boldsymbol{\psi}\|_{\Om}=0$. Since $\mZ^{-\beta}_T(1)\subset \mX_T(1)$ and  $\|\curl\cdot\|_{\Om}$ is a norm in $\mX_T(1)$ (Proposition \ref{PropoEmbeddingCla}), we deduce that $\boldsymbol{\psi}=0$. This shows that $\mathbb{B}_T^{\ast}$ is injective and that $\mathbb{B}_T$ is an isomorphism.
\end{proof}
\noindent Now we use the above lemma to prove the following result which is essential in the analysis of the Problem (\ref{MainPbH}) for the magnetic field. This is somehow an extension of the result of item $i)$ of Proposition \ref{propoPotential} for singular fields which are not in $\Lspace^2(\Om)$.
\begin{proposition}\label{propoPotentialWeight}
Under Assumption 1, for all $0\le\beta<1/2$, if $\u\in\boldsymbol{\mV}^0_{\beta}(\Om)$ satisfies $\div\,\u=0$ in $\Om$, then there exists a unique $\boldsymbol{\psi}\in\mZ^{\beta}_T(1)$ such that $\u= \curl \boldsymbol{\psi}$.
\end{proposition}
\begin{proof}
Let $\u\in \boldsymbol{\mV}^0_{\beta}(\Om)$ be such that $\div\,\u=0$ in $\Om$. According to Lemma \ref{LemmaWPPoids}, we know that there is a unique $\boldsymbol{\psi}\in\mZ^{\beta}_T(1)$ such that
\[
\int_{\Om}\curl\boldsymbol{\psi}\cdot\curl\overline{\boldsymbol{\psi}'}\,dx=\int_{\Om}\u\cdot\curl\overline{\boldsymbol{\psi}'}\,dx,\qquad\forall\boldsymbol{\psi}'\in\mZ^{-\beta}_T(1).
\]
Then we have 
\begin{equation}\label{FirstOrtho}
\int_{\Om}(\u-\curl\boldsymbol{\psi})\cdot\curl\overline{\boldsymbol{\psi}'}\,dx=0,\qquad\forall\boldsymbol{\psi}'\in\mZ^{-\beta}_T(1).
\end{equation}
Since $\u$ is divergence free in $\Om$, we also have 
\begin{equation}\label{SecondOrtho}
\int_{\Om}(\u-\curl\boldsymbol{\psi})\cdot\nabla \overline{p'}\,dx=0,\qquad\forall p'\in\mathring{\mV}^1_{-\beta}(\Om).
\end{equation}
Now if $\v$ is an element of $\boldsymbol{\mV}^0_{-\beta}(\Om)\subset\Lspace^2(\Om)$, from item $iv)$ of Proposition \ref{propoPotential}, we know that there holds the decomposition
\begin{equation}\label{DecompHInterm}
\v= \nabla p' +\curl \boldsymbol{\psi}',
\end{equation}
for some $p'\in \mH^1_{0}(\Omega)$ and some $\boldsymbol{\psi}'\in \mX_T(1)$. Taking the divergence in (\ref{DecompHInterm}), we get 
\begin{equation}\label{SolLapl}
\Delta p'= \div \,\v\in (\mathring{\mV}^1_{\beta}(\Om))^\ast.
\end{equation} 
From Proposition \ref{PropoLaplaceOp}, since $0\le\beta<1/2$, we know that (\ref{SolLapl}) admits a solution in $\mathring{\mV}^1_{-\beta}(\Om)\subset\mH^1_0(\Om)$. Using uniqueness of the solution of (\ref{SolLapl}) in $\mH^1_0(\Om)$, we obtain that $p'\in\mathring{\mV}^1_{-\beta}(\Om)$. This implies that $\curl\boldsymbol{\psi}'=\v-\nabla p'\in\boldsymbol{\mV}^0_{-\beta}(\Om)$ and so $\boldsymbol{\psi}'\in\mZ^{-\beta}_T(1)$. From (\ref{FirstOrtho}) and (\ref{SecondOrtho}), we infer that
\[
\int_{\Om}(\u-\curl\boldsymbol{\psi})\cdot \overline{\v}\,dx=0,\qquad\forall \v\in\boldsymbol{\mV}^0_{-\beta}(\Om).
\]
This shows that $\u=\curl\boldsymbol{\psi}$. Finally, if $\boldsymbol{\psi}_1$, $\boldsymbol{\psi}_2$ are two elements of $\mZ^{\beta}_T(1)$ such that $\u=\curl\boldsymbol{\psi}_1=\curl\boldsymbol{\psi}_2$, then $\boldsymbol{\psi}_1-\boldsymbol{\psi}_2$ belongs to $\boldsymbol{\mX}_T(1)$ and satisfies $\curl(\boldsymbol{\psi}_1-\boldsymbol{\psi}_2)=0$ in $\Om$. From Proposition \ref{PropoEmbeddingCla}, we deduce that $\boldsymbol{\psi}_1=\boldsymbol{\psi}_2$.
\end{proof}

\section{Energy flux of the singular function}
\label{Appendix energy flux}
\begin{lemma}\label{lemmaNRJ}
With the notations of (\ref{defSing}), we have
\[
\Im m\,\bigg(\int_{\Om}\div(\eps\nabla\overline{s^{+}})\,{s^+}\,dx\bigg) =\eta\int_{\mathbb{S}^2}\eps|\Phi|^2ds.
\]
\end{lemma}
\begin{proof}
Set $\Om_\tau:=\Om\setminus \overline{B(O,\tau)}$. Noticing that $\div(\eps\nabla\overline{s^{+}})$ vanishes in a neighbourhood of the origin, we can write
\[
\begin{array}{lcl}
\int_{\Om}\div(\eps\nabla\overline{s^{+}} )\,{s^+}\,dx&=&\lim_{\tau\rightarrow 0}\int_{\Om_\tau}\div(\eps\nabla\overline{s^{+}} )\,{s^+}\,dx\\
&=&
\lim_{\tau\rightarrow 0}\bigg(-\int_{\Om_\tau}\eps|\nabla s^+|^2dx-\int_{\partial B(O,\tau)}\eps\frac{\overline{\partial s^+}}{\partial r}s^+ds\bigg).
\end{array}
\]
Taking the imaginary part and observing that 
\[
\int_{\partial B(O,\tau)}\eps\frac{\overline{\partial s^+}}{\partial r}s^+ds=-\left(\frac{1}{2}+i\eta\right)\int_{\mathbb{S}^2}\eps|\Phi|^2ds,
\]
the result follows. 
\end{proof}

\section{Dimension of $\mX^{\mrm{out}}_N(\eps)/\mX_N(\eps)$}
\begin{lemma}\label{LemmaCodim}
Under Assumptions 1--3, we have $\dim\,(\mX^{\mrm{out}}_N(\eps)/\mX_N(\eps))=1$.
\end{lemma}
\begin{proof}
If $\u_1=c_1\nabla s^{+}+\tilde{\u}_1$, $\u_2=c_2\nabla s^{+}+\tilde{\u}_2$ are two elements of $\mX^{\mrm{out}}_N(\eps)$, then $c_2\u_1-c_1\u_2\in\mX_N(\eps)$, which shows that $\dim\,(\mX^{\mrm{out}}_N(\eps)/\mX_N(\eps))\le1$.\\
Now let us prove that $\dim\,(\mX^{\mrm{out}}_N(\eps)/\mX_N(\eps))\ge1$. Introduce $\tilde{\mathfrak{s}}\in\mathring{\mV}^{\mrm{out}}$ the function such that $A^{\mrm{out}}_{\eps}\tilde{\mathfrak{s}}=\div(\eps\nabla  s^-)$.  Note that since $\div(\eps\nabla  s^-)$ vanishes in a neighbourhood of the origin, it belongs to $(\mathring{\mV}^1_{\gamma}(\Om))^{\ast}$ for all $\gamma\in\R$. Then set 
\begin{equation}\label{defs}
\mathfrak{s}=s^-+\tilde{\mathfrak{s}}.
\end{equation}
Observe that $\mathfrak{s}\in\mathring{\mV}^1_{\gamma}(\Om)$ for all $\gamma>0$ and that $\div(\eps\nabla \mathfrak{s})=0$ in $\Om\setminus\{O\}$ ($\mathfrak{s}$ is a non zero element of $\ker\,A^{\gamma}_{\eps}$ for all $\gamma>0$). Let $\tilde{\u}\in(\mathscr{C}^{\infty}_0(\Om\setminus\{O\}))^3$ be a field such that $\textstyle\int_{\Om}\eps\tilde{\u}\cdot\nabla\overline{\mathfrak{s}}\,dx\ne0$. The existence of such a $\tilde{\u}$ can be established thanks to the density of $(\mathscr{C}^{\infty}_0(\Om\setminus\{O\}))^3$ in $\Lspace^2(\Omega)$, considering for example an approximation of $\mathbbm{1}_{\overline{B}}\nabla \mathfrak{s}\in\Lspace^2(\Omega)$ where $\mathbbm{1}_{\overline{B}}$ is the indicator function of a ball included in $\mathcal{M}$. Introduce $\zeta=c\,s^++\tilde{\zeta}\in\mathring{\mV}^{\mrm{out}}$, with $c\in\Cplx$, $\tilde{\zeta}\in\mathring{\mV}^1_{-\beta}(\Om)$, the function such that $A^{\mrm{out}}_{\eps}\zeta=-\div(\eps\tilde{\u})$. This is equivalent to have
\[
-c\int_{\Om}\div(\eps\nabla s^+)\overline{\varphi'}\,dx+\int_{\Om}\eps\nabla\tilde{\zeta}\cdot\nabla\overline{\varphi'}\,dx=\int_{\Om}\eps\tilde{\u}\cdot\nabla\overline{\varphi'}\,dx,\qquad \forall\varphi'\in\mathring{\mV}^1_{\beta}(\Om).
\]
Clearly $\nabla\zeta-\tilde{\u}=c\nabla s^++(\nabla\tilde{\zeta}-\tilde{\u})$ is an element of $\mX^{\mrm{out}}_N(\eps)$. Moreover taking $\varphi'=\mathfrak{s}$ above, we get
\[
-c\int_{\Om}\div(\eps\nabla s^+)\overline{\mathfrak{s}}\,dx=\int_{\Om}\eps\tilde{\u}\cdot\nabla\overline{\mathfrak{s}}\,dx\ne0.
\]
This shows that $c\ne0$ and guarantees that $\dim\,(\mX^{\mrm{out}}_N(\eps)/\mX_N(\eps))\ge1$.
\end{proof}

\bibliographystyle{plain} 
\bibliography{Biblio}

\end{document}